\documentclass[12pt,reqno]{article}
\usepackage[english]{babel}
\usepackage{amsmath,amsthm}
\usepackage{amsfonts}
\usepackage{graphicx}
\usepackage{enumerate}
\usepackage[usenames,dvipsnames]{color}
\usepackage{subfigure}
\usepackage[right,pagewise,displaymath, mathlines]{lineno}
\usepackage{epstopdf}
\usepackage{color}
\usepackage{multirow}

\numberwithin{equation}{section}
\numberwithin{figure}{section}
\numberwithin{table}{section}

\newtheorem{theorem}{Theorem}[section]

\newtheorem{lemma}{Lemma}[section]

\newtheorem{proposition}{Proposition}[section]

\numberwithin{equation}{section}
\begin{document}
%\linenumbers

\begin{center}

{\Large\bf Estimating the index of increase via balancing deterministic and random data}

\vspace*{7mm}

{\large Lingzhi Chen$^{1,}\footnote{lchen522@uwo.ca}$,
Youri Davydov$^{2,}\footnote{youri.davydov@univ-lille1.fr}$,
Nadezhda Gribkova$^{3,}\footnote{n.gribkova@spbu.ru}$,
\break and
Ri\v cardas Zitikis$^{1,}\footnote{rzitikis@uwo.ca}$}

\bigskip

$^{1}$\textit{School of Mathematical and Statistical Sciences,
Western University, \break London, Ontario N6A 5B7, Canada}

\medskip

$^{2}$\textit{Chebyshev Laboratory, St.\,Petersburg State University, Vasilyevsky Island, \break St.\,Petersburg 199178, Russia}

\medskip

$^{3}$\textit{Faculty of Mathematics and Mechanics, St.\,Petersburg State University,\\ St.\,Petersburg 199034, Russia}

\end{center}

\medskip

\begin{quote}
\textbf{Abstract.} We introduce and explore an empirical index of increase that works in both deterministic and random environments, thus allowing to assess monotonicity of functions that are prone to random measurement-errors. We  prove consistency of the index and show how its rate of convergence is influenced by deterministic and random parts of the data. In particular, the obtained results suggest a frequency at which observations should be taken in order to reach any pre-specified level of estimation precision. We illustrate the index using data arising from purely deterministic and error-contaminated functions, which may or may not be monotonic.

\medskip

{\it Key words and phrases:} index of increase, determinism, randomness, measurement errors, smoothing, cross validation.

\medskip

{\it 2010 MSC:}  Primary: 62G05, 62G08, 62G20; Secondary: 62P15, 62P20, 62P25.
\end{quote}

\newpage

\section{Introduction}
\label{intro}

Dynamic processes in populations are often described using functions (e.g., Bebbington et al., 2007, 2011; and references therein). They are observed in the form of data points, usually contaminated by measurement errors. We may think of these points as randomly perturbed true values of underlying functions, whose measurements are taken at certain time instances. The functions, their rates of change, and de/acceleration can be and frequently are  non-monotonic. Nevertheless, it is of interest to assess and even compare the extent of their monotonicity, or lack of it. We refer to Qoyyimi (2015) for a discussion and literature review of various applications.

Several methods for assessing monotonicity have been suggested in the literature (e.g., Davydov and Zitikis, 2005, 2017;  Qoyyimi, and Zitikis, 2014, 2015). In particular, Davydov and Zitikis (2017) show the importance of such assessments in insurance and finance, especially when dealing with weighted insurance calculation principles (Furman and Zitikis, 2008), among which we find such prominent examples as the Esscher (B\" uhlmann, 1980, 1984), Kamps (1998), and Wang (1995, 1998) premiums. Furthermore, Egozcue et al. (2011) provide problems in economics where the sign of the covariance
\begin{equation}\label{cov-1}
\mathbf{Cov}[X,w(X)]
\end{equation}
needs to be determined for various classes of function $w$. One of such examples concerns the slope of indifference curves in two-moment expected utility theory (e.g., Eichner and Wagener, 2009;  Sinn, 1990;  Wong, 2006; and references therein). Another problem concerns decision making (e.g., speculation, normal backwardation, contango, etc.) of competitive companies under price uncertainty (e.g., Holthausen, 1979; Feder et al., 1980; Hey, 1981; Meyer and Robison, 1988; and references therein).

Lehmann (1966) has shown that if the function $w$ is monotonic, then covariance (\ref{cov-1}) is either positive (when $w$ is increasing) or negative (when $w$ is decreasing). This monotonicity assumption on $w$, though satisfied in a number of cases of practical interest, excludes a myriad of important cases with more complex risk profiles. For example, when dealing with the aforementioned economics-based problems, the role of $w$ is played by the derivative $u'$ of the underlying utility function, which may not be convex or concave everywhere, as argued and illustrated by, e.g., Friedman and Savage (1948), Markowitz (1952), Kahneman and Tversky (1979), Tversky and Kahneman (1992), among others. Hence, since $w$ might be non-monotonic, how far can this function be from being monotonic, or increasing? Furthermore, since the population risk- or utility-profile cannot be really known, the non-monotonicity of $w$ needs to be assessed from data, and this leads us to the statistical problem of this paper.

In addition, supported by the examples of Anscombe (1973) on potential pitfalls when using the classical correlation coefficient, Chen and Zitikis (2017) argue in favour of using the index of increase, as defined by Davydov and Zitikis (2017), for assessing non-monotonicity of scatterplots. Chen and Zitikis (2017) apply this approach to analyze and compare student performance in subjects such as mathematics, reading and spelling, and illustrate their reasoning on data provided by Thorndike and Thorndike-Christ (2010). One of the methods discussed by Chen and Zitikis (2017) deals with scatterplots representing finite populations, in which case large-sample estimation is not possible. The other method involves large-sample regression techniques (Figure \ref{fig-loess}), in which case Chen and Zitikis (2017) calculate the corresponding indices of increase using a numerical approach, that gives rise to the values denoted by $\mathrm{I}$ and reported in the bottom-right corners of the panels of Figure \ref{fig-loess}.
\begin{figure}[h!]
\centering
\includegraphics[width = 0.7\textwidth]{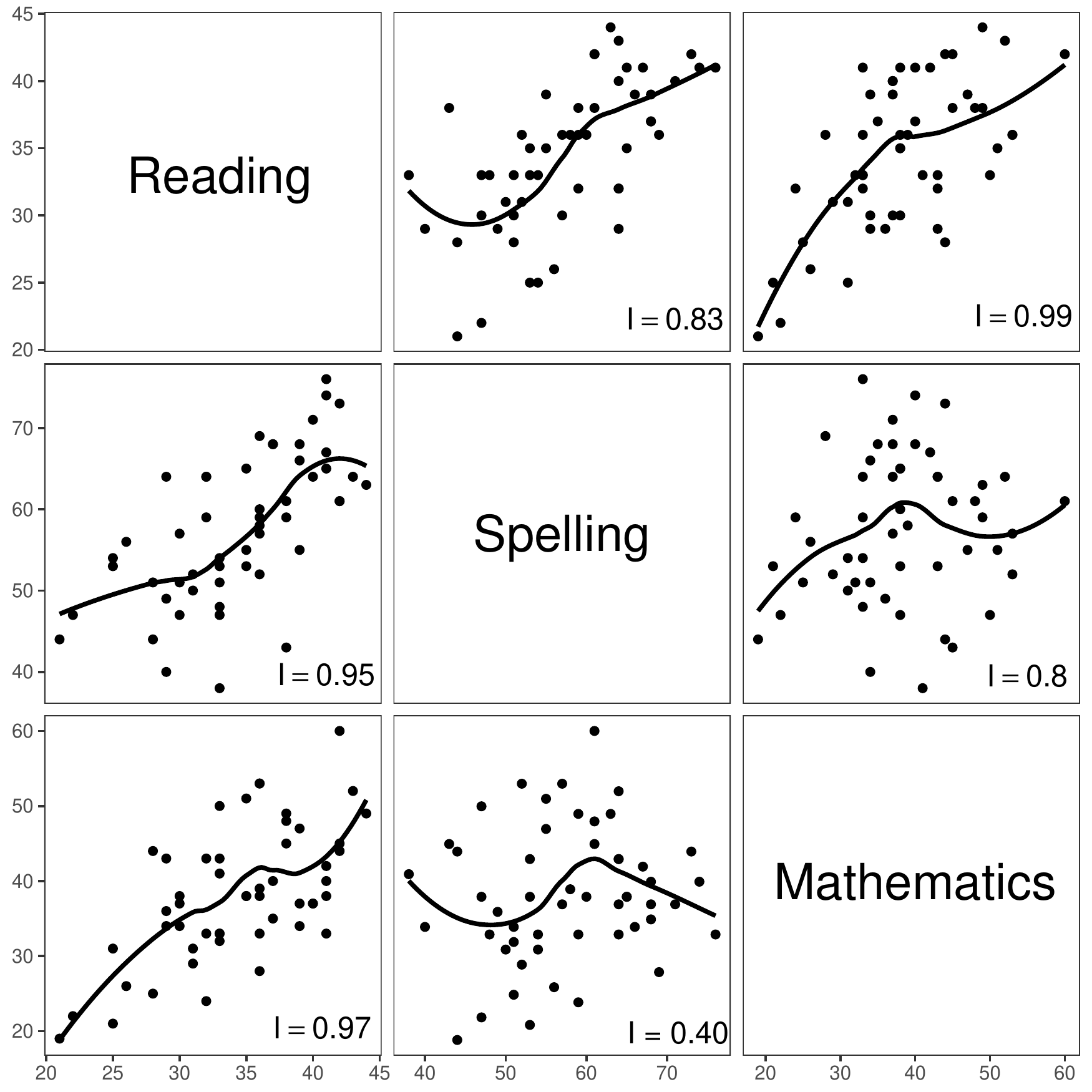}
\caption{Regression curves fitted to the student scores reported by Thorndike and Thorndike-Christ (2010), and their indices of increase.}
\label{fig-loess}
\end{figure}
Though important, these methods do not allow direct large-sample  non-monotonicity  quantifications and thus inferences about larger populations. In this paper, therefore, we offer a statistically attractive and computationally efficient procedure for assessing data patterns that arise from non-monotonic patterns contaminated by random measurement errors.

We have organized the rest of the paper as follows. In Section \ref{tech}, we introduce the index and provide basic arguments leading to it. In Section \ref{practice}, we explain why and how the index needs to be adjusted in order to become useful in situations when random measurement errors are present. In Section \ref{consistency}, we rigorously establish consistency of the estimator and introduce relevant data-exploratory and cross-validatory techniques.  Since the limiting distribution of the estimator is complex, in Section \ref{bootstrap} we implement a bootstrap-based procedure for determining standard errors and, in turn, for deriving confidence intervals. Section~\ref{conclude} concludes the paper with a brief summary of our main contributions.

\section{The index of increase}
\label{tech}

Davydov and Zitikis (2017) have introduced the index of increase
\begin{equation}
\mathrm{I}(h_0)={\int_a^b (h_0')_{+}\text{d}\lambda \over \int_a^b |h_0'|\text{d}\lambda }
\quad \bigg (:={\int_a^b (h_0'(t))_{+}\text{d}t \over \int_a^b |h_0'(t)|\text{d}t } \bigg )
\label{index}
\end{equation}
for any absolutely continuous (e.g., differentiable) function $h_0$ on interval $[a,b]$, where $(h_0')_{+}:=\max \{h_0' , 0\} $, and ``$:=$'' denotes equality by definition. Throughout the paper, we use $\lambda$ to denote the Lebesgue measure, which helps us to write integrals compactly, as seen from the  ratios above. We shall explain how the index arises later in the current section. Of course, this framework reduces to the unit interval $[0,1]$ by considering the function $h(t):=h_0(a+(b-a)t)$ instead of $h_0$. Namely, we have
\begin{equation}
\mathrm{I}(h_0)={\int_0^1 (h')_{+}\text{d}\lambda \over \int_0^1 |h'|\text{d}\lambda }=:\mathrm{I}(h).
\label{index-norm}
\end{equation}

To illustrate, in Figure \ref{fig-quartet}
\begin{figure}[h!]
\centering
  \centering
  \subfigure[$\mathrm{I}(h_{1})\approx \mathrm{I}_n(h_{1})=0.6667$]{%
    \includegraphics[height=0.3\textwidth]{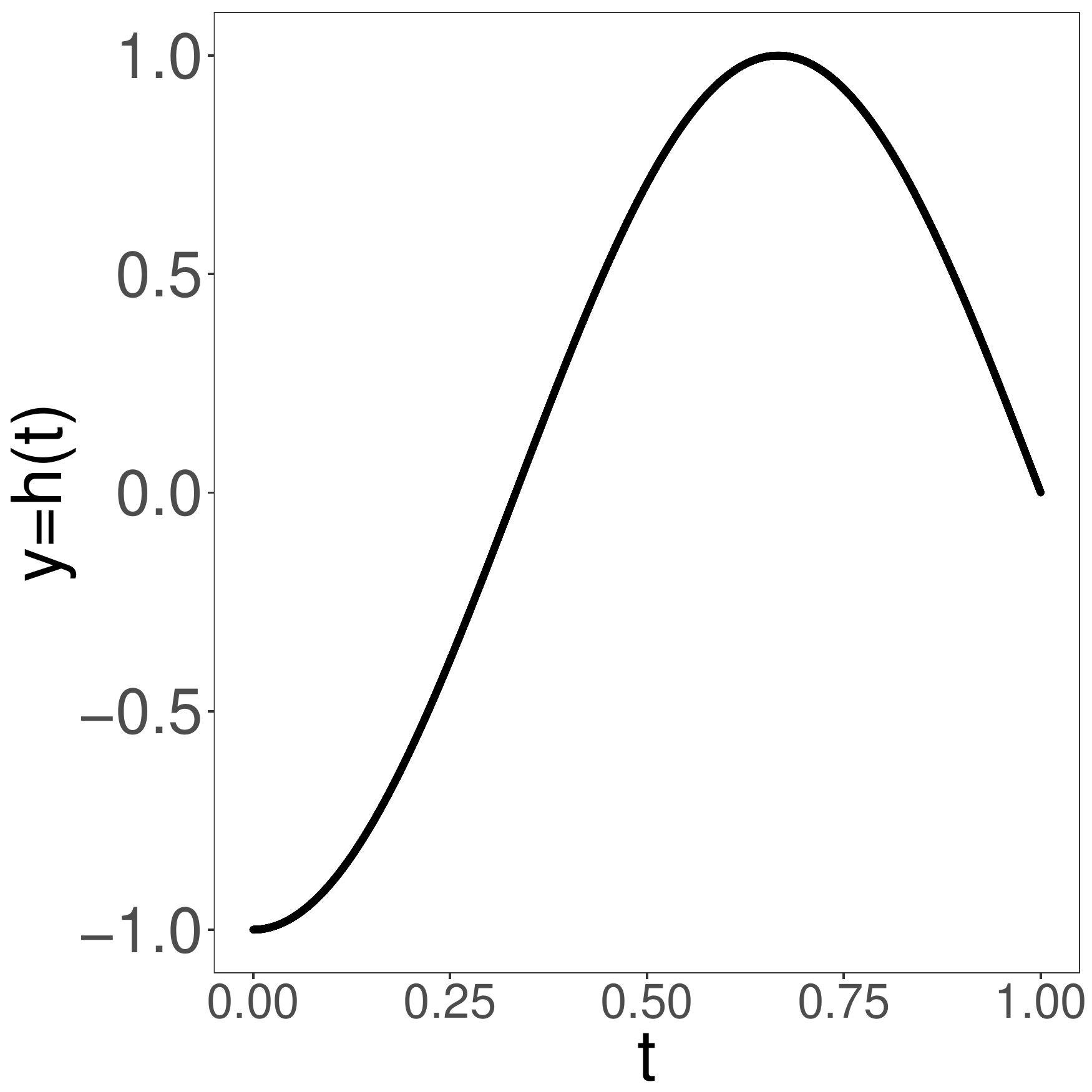}}%
\hspace{0.5in}
\subfigure[$\mathrm{I}(h_{2})\approx\mathrm{I}_n(h_{2})=0.3333$]{%
    \includegraphics[height=0.3\textwidth]{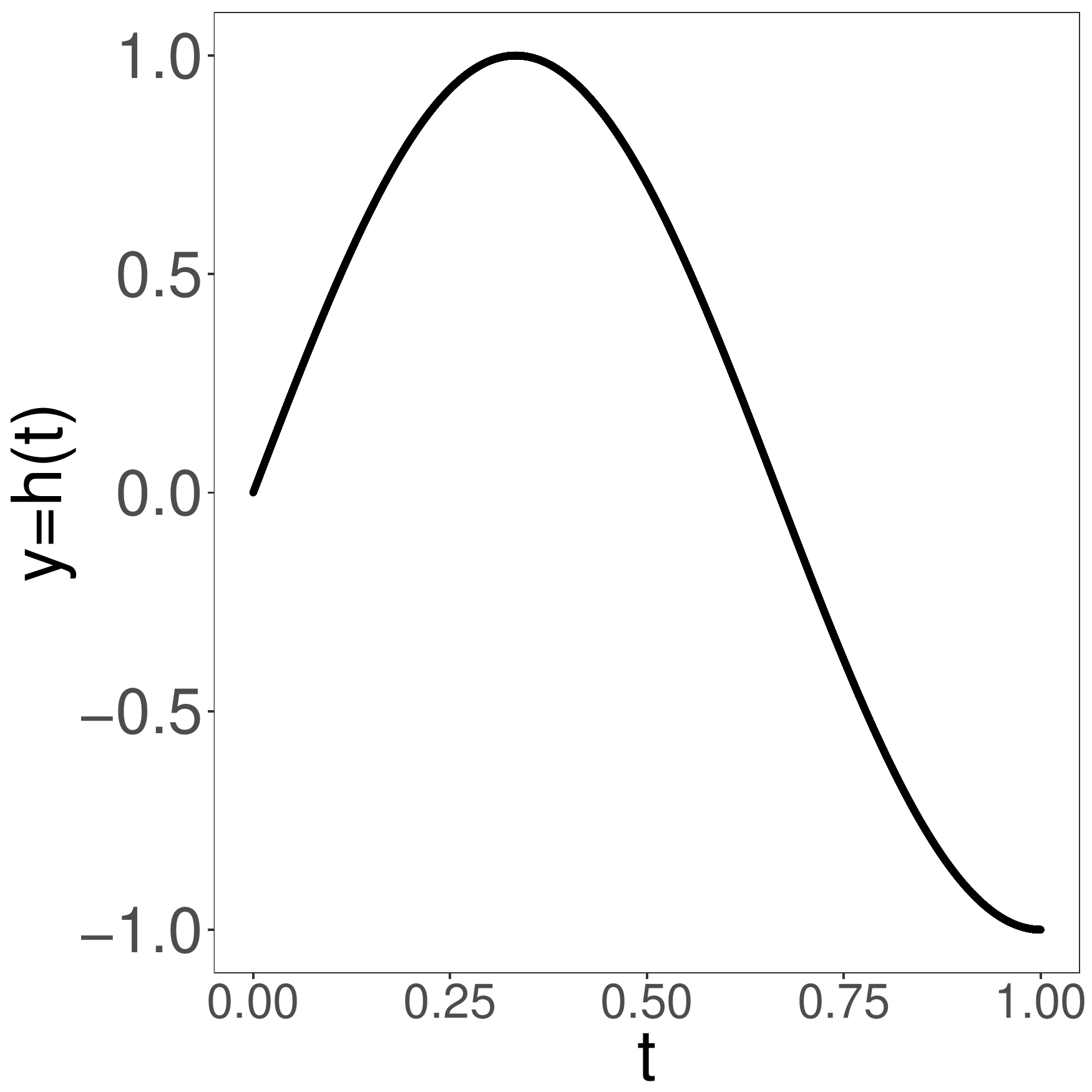}}%
\hspace{0.5in}
 \subfigure[$\mathrm{I}(h_{3})=\mathrm{I}_n(h_{3})=1$]{%
    \includegraphics[height=0.3\textwidth]{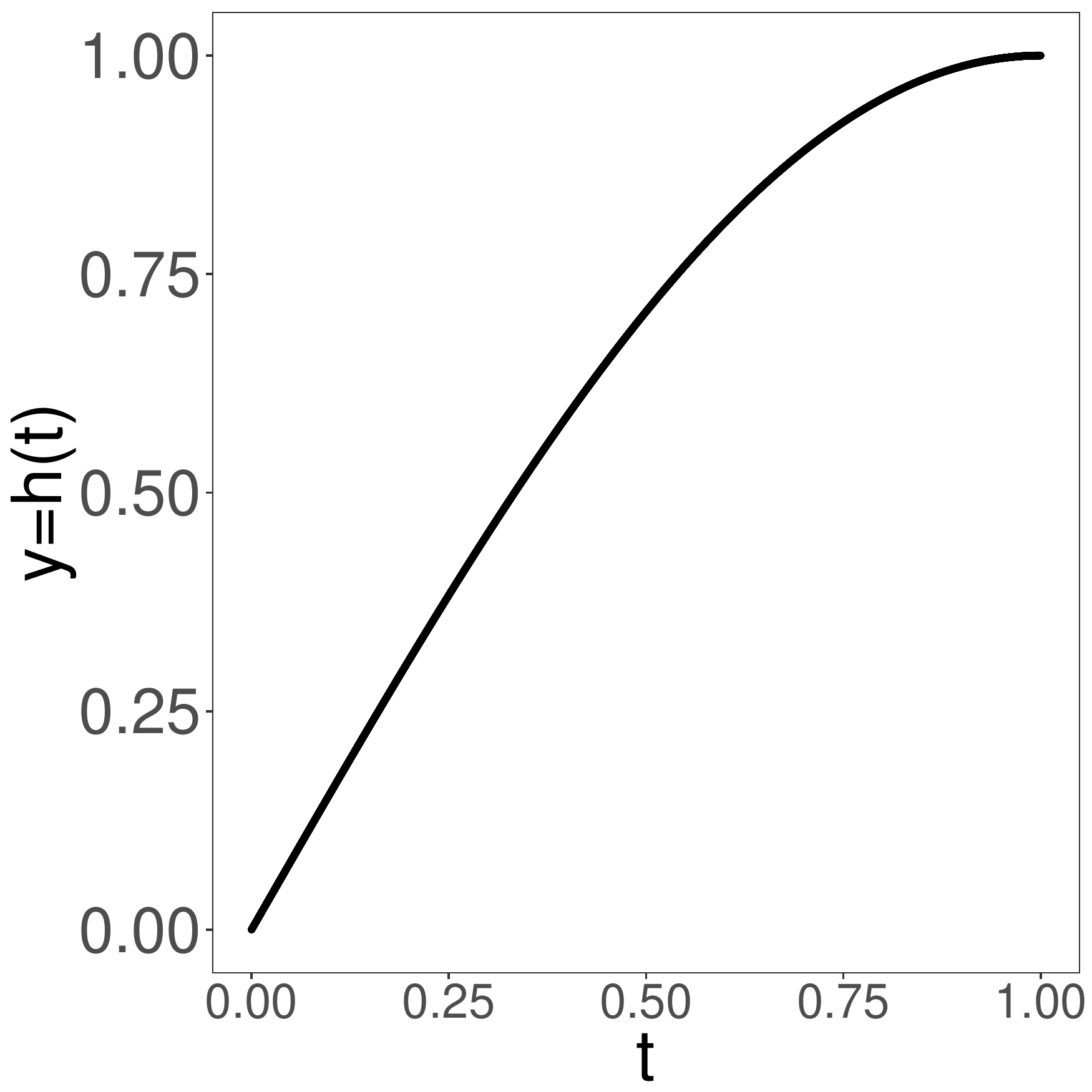}}%
\hspace{0.5in}
\subfigure[$\mathrm{I}(h_{4})=\mathrm{I}_n(h_{4})=0$]{%
    \includegraphics[height=0.3\textwidth]{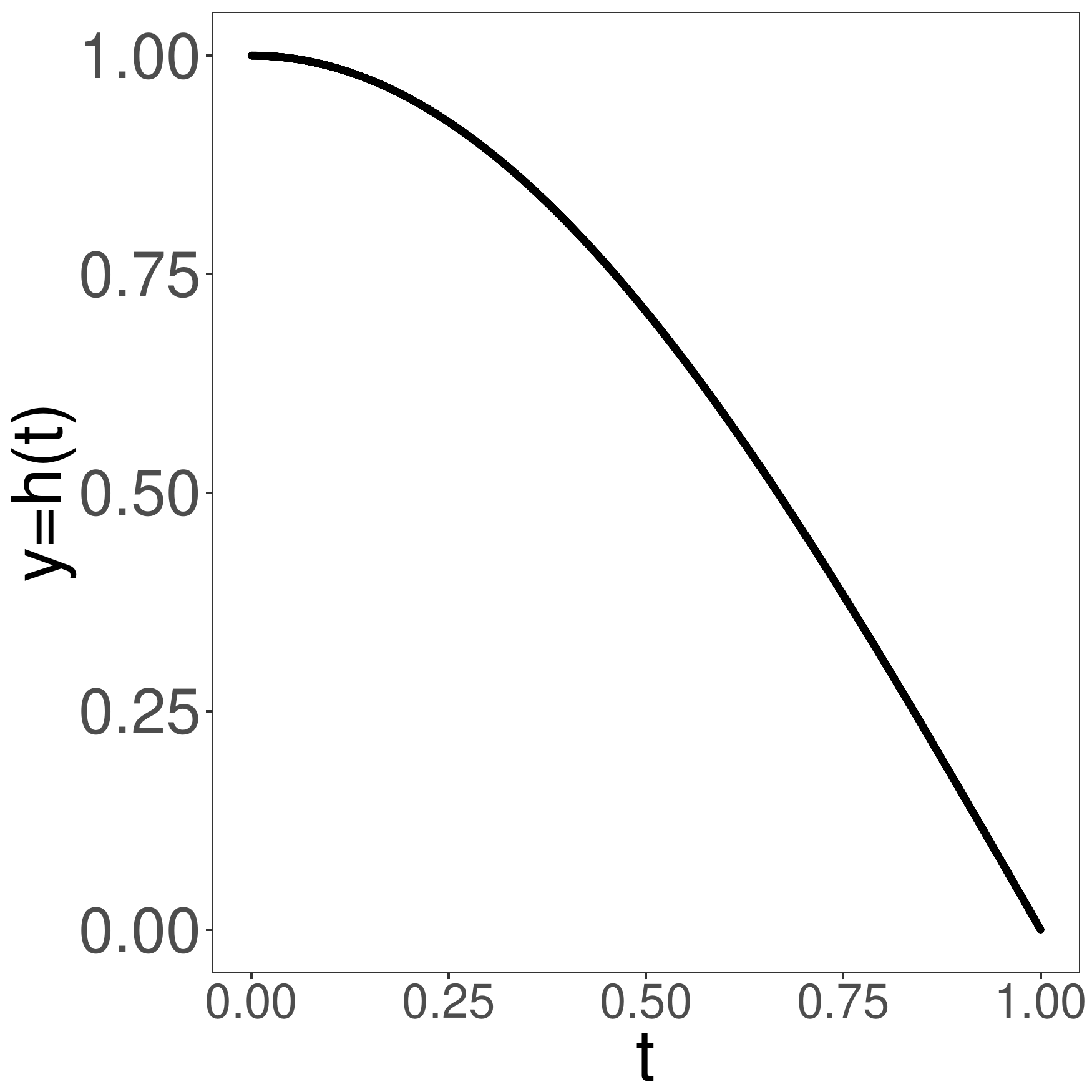}}%
\caption{The functions of quartet (\ref{quartet-1}) and their indices of increase}
\label{fig-quartet}
\end{figure}
we have visualized the following quartet of functions
\begin{equation}
\begin{split}
h_1(t)=\sin\Big(-{\pi\over 2}+{3\pi \over 2}t\Big),
&\quad h_2(t)=\cos\Big(-{\pi\over 2}+{3\pi \over 2}t\Big),
\\
h_3(t)=\sin\Big({\pi\over 2}t\Big),
&\quad h_4(t)=\cos\Big({\pi\over 2}t\Big),
\end{split}
\label{quartet-1}
\end{equation}
and we have also calculated their indices of increase. Since $h_3$ and $h_4$ are monotonic functions on the interval $[0,1]$, calculating their indices of increase using formula (\ref{index-norm}) is trivial, but the same task in the case of non-monotonic functions $h_1$ and $h_2$ requires some effort. To facilitate such calculations in a speedy fashion, and irrespective of the complexity of functions, we suggest using the numerical approximation
\begin{equation}\label{approx-numeric}
\mathrm{I}_n(h):={\sum_{i=2}^n (h(t_{i,n})-h(t_{i-1,n}))_{+} \over \sum_{i=2}^n |h(t_{i,n})-h(t_{i-1,n})| }
\end{equation}
with $t_{i,n}=(i-1)/(n-1)$ for $i=1,\dots, n$. Intuitively, $\mathrm{I}_n(h)$ is the proportion of the upward movements of the function $h$ with respect to all the movements, upward and downward.

Knowing the convergence rate of $\mathrm{I}_n(h)$ to $\mathrm{I}(h)$ when $n\to \infty $ is important as it allows us to set a frequency $n$ at which the measurements of $h(t_{i,n})$ could be taken during the observation period (e.g., unit interval $[0,1]$) so that any pre-specified estimation precision of $\mathrm{I}(h)$ would be achieved. For example, we have used $n=10000$ to calculate the index values with the four-digit precision reported in Figure \ref{fig-quartet}. We refer to Chen and Zitikis (2017) for details on computational precision.

The following proposition, which is a special case of Lemma \ref{le-a2} below, establishes the convergence rate based on the level of smoothness of the function $h$.

\begin{proposition}
\label{prop-a3}
Let $h$ be a differentiable function defined on the unit interval $[0,1]$, and let its derivative $h'$ be $\gamma$-H\"{o}lder continuous for some $\gamma\in (0,1]$. Then, when $n\to \infty $,  we have
\begin{equation}\label{partition-1a}
\sum_{j=2}^n \ell \Big (h(t_{i,n})-h(t_{i-1,n})\Big )
= \int_0^{1} \ell \big ( h' \big )\text{d}\lambda + O(n^{-\gamma})
\end{equation}
for any positively homogeneous and Lipschitz function $\ell$ (e.g., $\ell(t)=t_{+}$ and $\ell(t)=|t|$). Consequently,
\begin{equation}\label{conv-rate}
\mathrm{I}_n(h)=\mathrm{I}(h)+O(n^{-\gamma}).
\end{equation}
\end{proposition}

To explain the basic meaning of the index $\mathrm{I}(h)$, we start with an un-normalized version of it, which we denote by $\mathrm{J}(h)$. Namely, let $ \mathcal{F}$ denote the set of all absolutely continuous functions $f$ on the interval $[0,1]$ such that $f(0)=0$. Denote the total variation
of $f\in \mathcal{F}$ on the interval $[0,1]$ by $\Vert f \Vert $, that is, $\Vert f \Vert =\int_0^1 |f'|\mathrm{d}\lambda $. Furthermore, by definition, we have $(f')_{+}=\max \{f' , 0\} $ and $(f')_{-}=\max \{ -f', 0\}$, and we also have the equations $f'=(f')_{+}-(f')_{-}$ and  $|f'|=(f')_{+}+(f')_{-}$. Finally, we use $\mathcal{F}^{-} $ to denote the set of all the functions $f\in \mathcal{F}$ that are non-increasing. All of these are of course well-known fundamental notions of Real Analysis (e.g., Kolmogorov and Fomin, 1970; Dunford and Schwartz,1988; and Natanson, 2016).

For any function $h\in \mathcal{F}$, we define its (un-normalized) index of increase $\mathrm{J}(h)$ as the distance between $h$ and the set $\mathcal{F}^{-} $, that is,
\begin{equation}
\mathrm{J}(h)
=\inf_{f\in \mathcal{F}^{-} }  \Vert h-f\Vert .
\label{def-00}
\end{equation}
Obviously, if $h$ is non-increasing, then $\mathrm{J}(h)=0$, and the larger the value of $\mathrm{J}(h)$, the farther the function $h$ is from being non-increasing on the interval $[0,1]$. Determining the index $\mathrm{J}(h)$ using its definition (\ref{def-00}) is not, however, a straightforward task, and to facilitate it, we next establish a very convenient integral representation of $\mathrm{J}(h)$.

\begin{theorem}[Davydov and Zitikis, 2017]\label{th-a1}
The infimum in definition (\ref{def-00}) is attained at any function
$f_1\in \mathcal{F}^{-}$ such that $f_1'=-(h')_{-}$, and thus
\begin{equation}\label{loi-2}
\mathrm{J}(h)=\int_0^1 (h')_{+}\mathrm{d}\lambda .
\end{equation}
\end{theorem}

A direct proof of this theorem was not provided by Davydov and Zitikis (2017), who refer to a more general and abstract result. Nevertheless, a short and enlightening proof exists, and we present it next.

\begin{proof}[Proof of Theorem \ref{th-a1}]
We start with the note that the bound $\mathrm{J}(h)\le \Vert h - f \Vert $ holds for every
function $f\in \mathcal{F}^{-}$, and in particular for the function
$f_1$ specified in the formulation of the theorem. Hence,
\begin{align}
\mathrm{J}(h)
&\le \int_0^1 |h'-f_1'|\mathrm{d}\lambda
\notag
\\
&= \int_0^1 |h'+(h')_{-}|\mathrm{d}\lambda
\notag
\\
&= \int_0^1 (h')_{+}\mathrm{d}\lambda .
\label{u-3}
\end{align}
It now remains to show the opposite bound. Let $T^{+}$ be the set of all
$t\in [0,1]$ such that $h'(t)> 0$, and let  $T^{-}$ be the complement of
the set  $T^{+}$, which consists of all those $t\in [0,1]$ for which
$h'(t)\le 0$. Then
\begin{align}
\mathrm{J}(h)
&=\inf_{f\in \mathcal{F}^{-}} \bigg ( \int_{T^{+}} |h'-f'|\mathrm{d}\lambda
+\int_{T^{-}} |h'-f'|\mathrm{d}\lambda  \bigg )
\notag
\\
&\ge \inf_{f\in \mathcal{F}^{-}} \int_{T^{+}} |h'-f'|\mathrm{d}\lambda
\notag
\\
&= \inf_{f\in \mathcal{F}^{-}} \bigg (\int_{T^{+}} h'\mathrm{d}\lambda
+ \int_{T^{+}} |f'|\mathrm{d}\lambda \bigg )
\notag
\\
&= \int_0^1 (h')_{+}\mathrm{d}\lambda ,
\label{d-3}
\end{align}
where the last equation holds when $f'(t)=0$ for all $t\in T^{+}$, that is, when $f'=-(h')_{-}$.  Bounds (\ref{u-3}) and (\ref{d-3}) establish equation (\ref{loi-2}), thus finishing the proof of Theorem \ref{th-a1}.
\end{proof}

The index $\mathrm{J}(h)$ never exceeds $\Vert h\Vert $, and so the normalized version of $\mathrm{J}(h)$ is
\[
\mathrm{I}(h):=\mathrm{J}(h)/\Vert h\Vert,
\]
which is exactly the index of increase given by equation (\ref{index-norm}). In summary, the index of increase $\mathrm{I}(h)$ is the normalized distance of the function $h$ from the set $\mathcal{F}^{-}$ of all non-increasing functions on the interval $[0,1]$: we have  $\mathrm{I}(h)=0$ when the function $h$ is non-increasing, and $\mathrm{I}(h)=1$ when the function is non-decreasing. The closer the index $\mathrm{I}(h)$ is to $1$, the more (we say) the function $h$ is increasing, and the closer it is to $0$, the less (we say) the function $h$ is increasing or, equivalently, the more it is decreasing.

\section{Practical issues and their resolution}
\label{practice}

Measurements are usually taken with errors, whose natural model is some distribution (e.g., normal) with mean $0$ and finite variance $\sigma^2$. In other words, the numerical index $\mathrm{I}_n(h)$ turns into the random index of increase
\begin{equation}\label{approx-random}
\mathrm{I}_n(h,\varepsilon):={\sum_{i=2}^n (Y_{i,n}-Y_{i-1,n})_{+} \over \sum_{i=2}^n |Y_{i,n}-Y_{i-1,n}| },
\end{equation}
where, for $i=1,\dots, n$,
\begin{equation}\label{approx-random-e}
Y_{i,n}=h(t_{i,n})+\varepsilon_i .
\end{equation}

Right at the outset, however, serious issues arise. To illustrate them in a speedy and transparent manner, we put aside mathematics such as in Davydov and Zitikis (2004, 2007) and, instead, simulate $n=10000$ standard normal errors $\varepsilon_i$, thus obtaining four sequences $Y_{i,n}$ corresponding to the functions of quartet (\ref{fig-quartet}).  Then we calculate the corresponding indices of increase using formula (\ref{approx-random}). All of the obtained values of $\mathrm{I}_n(h)$ are virtually equal to $1/2$ (see Figure \ref{fig-quartet-error}).
\begin{figure}[h!]
\centering
  \subfigure[$\mathrm{I}(h_{1})= 0.6667$, $\mathrm{I}_n(h_1,\varepsilon)= 0.5$]{%
    \includegraphics[height=0.35\textwidth]{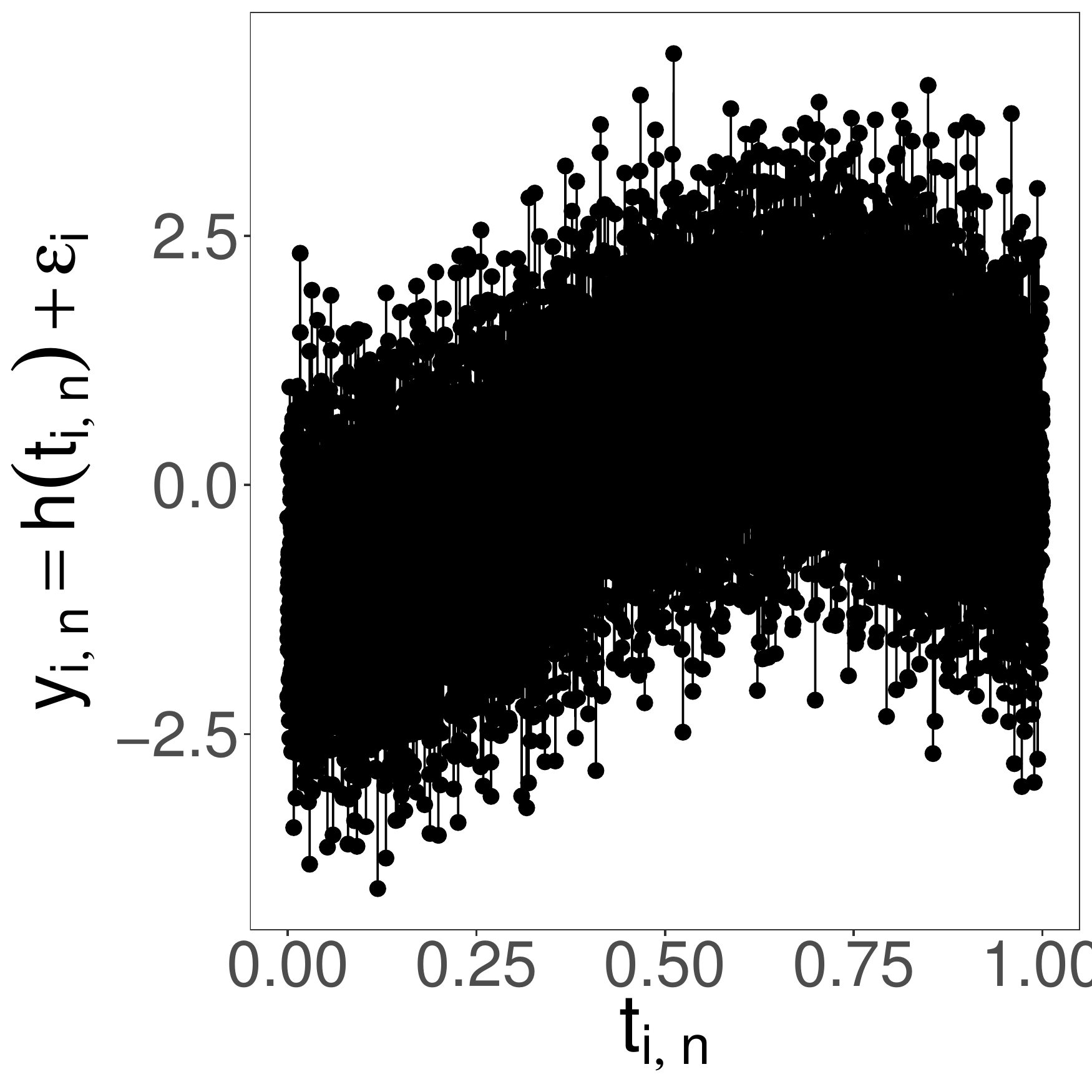}}%
\hspace{0.5in}
\subfigure[$\mathrm{I}(h_{2})= 0.3333$, $\mathrm{I}_n(h_2,\varepsilon)= 0.5$]{%
    \includegraphics[height=0.35\textwidth]{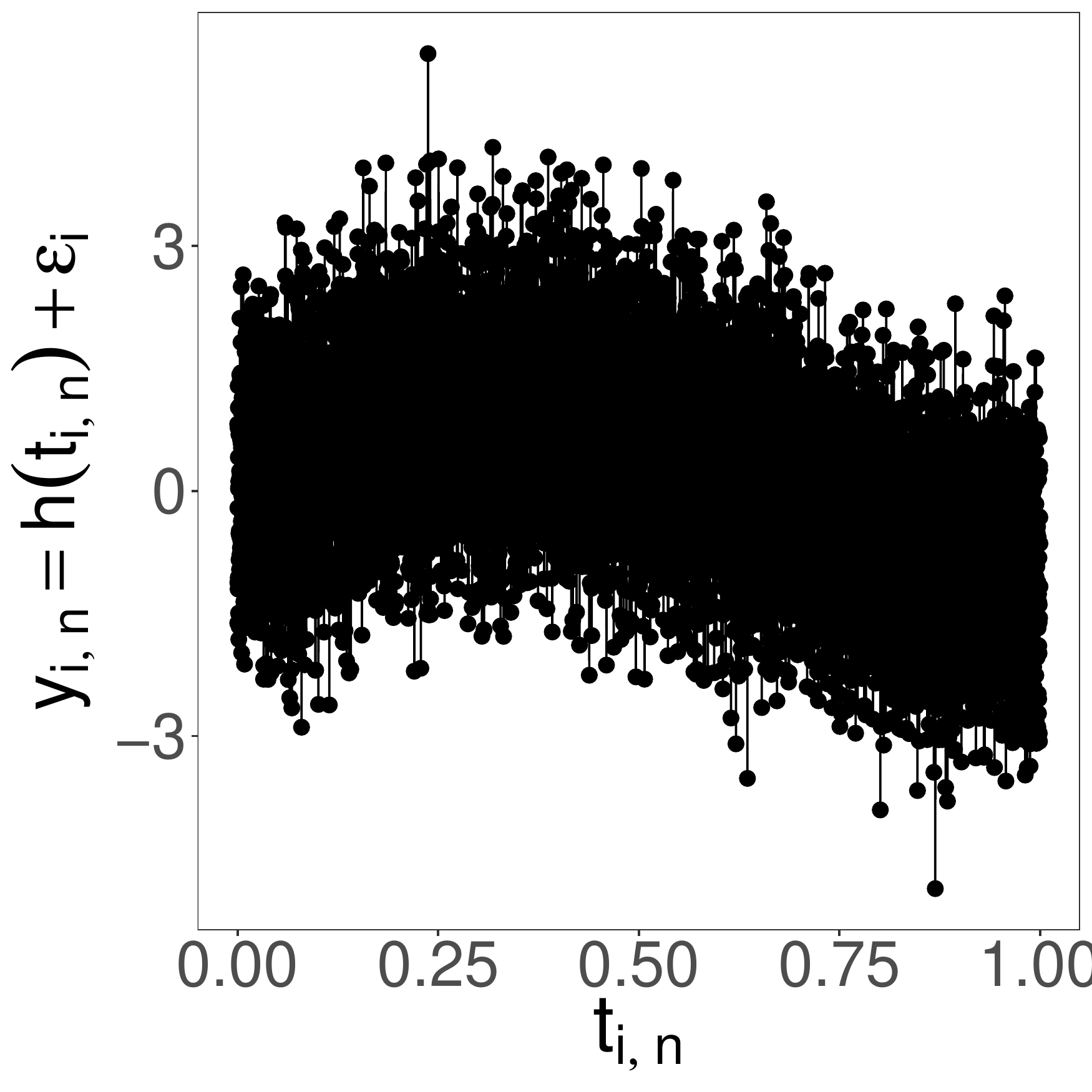}}%
\hspace{0.5in}
 \subfigure[$\mathrm{I}(h_{3})= 1$, $\mathrm{I}_n(h_3,\varepsilon)= 0.5$]{%
    \includegraphics[height=0.35\textwidth]{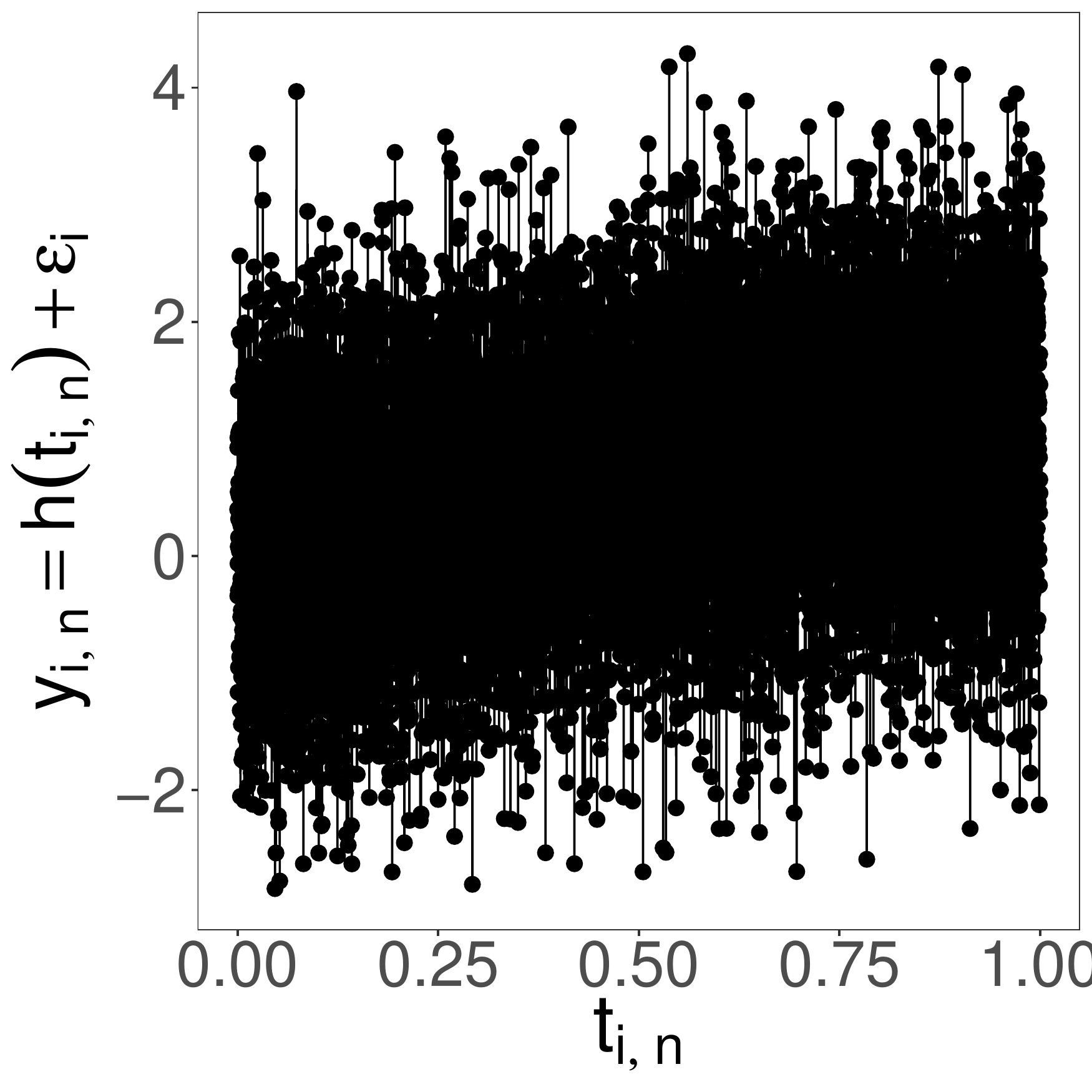}}%
\hspace{0.5in}
\subfigure[$\mathrm{I}(h_{4})= 0$, $\mathrm{I}_n(h_4,\varepsilon)= 0.5$]{%
    \includegraphics[height=0.35\textwidth]{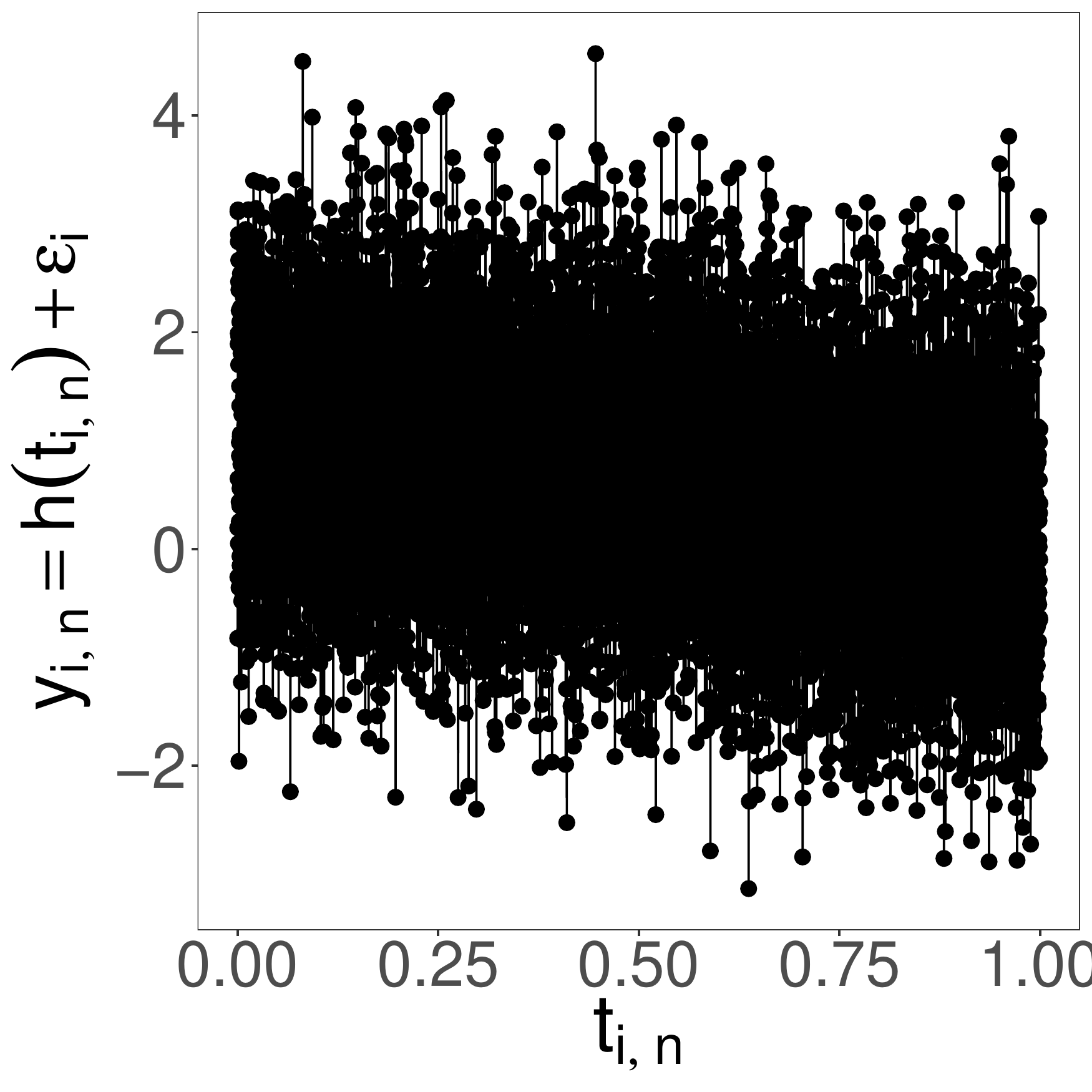}}%
\caption{The indices of increase and their numerical estimators for quartet (\ref{quartet-1}) with added random errors.}
\label{fig-quartet-error}
\end{figure}
Clearly, there is something amiss.

It is not, however, hard to understand the situation: when all $\varepsilon_i$'s are zero, the definition of the integral as the limit of the Riemann sums works as intended, but when the $\varepsilon_i$'s are not zero, they accumulate so much when $n$ gets larger that the deterministic part (i.e., the Riemann sum) gets hardly, if at all, visible (compare Figures \ref{fig-quartet} and \ref{fig-quartet-error}). In summary, we are facing two extremes:
\begin{itemize}
  \item If the model is purely deterministic in the sense that there are no measurement errors, which we can understandably argue to be outside the realm of practice, then the more frequently we observe the function $h$, the more precisely we can estimate its index of increase.
  \item If, however, there are measurement errors, as they usually are in practice, then the more frequently we observe the function, the less precisely we can estimate   its index of increase, because the accumulated measurement errors obscure the deterministic part.
\end{itemize}
Neither of the two extremes can be of much interest, or use, for reasons either practical or computational. The purpose of this paper is to offer a way out of this difficulty by showing how to strike a good balance between determinism and randomness inherent in the problem.

We next present an intuitive consideration that will guide our subsequent mathematical considerations, and it will also hint at potential applications of this research. Namely, suppose that the unit interval $[0,1]$ represents an one-day observation period, and let an observation be taken (e.g., by a measuring equipment) every second. Hence, in total, we have $n=86400$ observations $Y_{i,n}$ of the (unknown) function $h$, and they are prone to measurement errors $\varepsilon_i$ as in expression (\ref{approx-random-e}). For the sake of argument, let $\varepsilon_i$'s be i.i.d.~standard normal. If we calculate the index $\mathrm{I}_n(h,\varepsilon)$ based on these data, we already know the problem: $\mathrm{I}_n(h,\varepsilon)$ tends to $ 1/2$ when $n\to \infty $. To diminish the influence of these errors, we average the observed values:
\begin{align*}
{1\over n}\sum_{i=1}^n Y_{i,n}
&={1\over n}\sum_{i=1}^n h(t_{i,n})+{1\over n}\sum_{i=1}^n \varepsilon_i
\\
&\stackrel{d}{\approx} \int_0^1 h\text{d}\lambda +{1\over \sqrt{n}}\varepsilon_0,
\end{align*}
where $\stackrel{d}{\approx}$ means  `approximately in distribution,' and $\varepsilon_0 $ follows the standard normal distribution. However, in the process of averaging out the errors, we have inevitably also averaged the deterministic part and arrived at the mean value $\int_0^1 h\text{d}\lambda$ of the function $h$. This value has very little to do with the index $\mathrm{I}(h)$, which fundamentally relies on the derivative $h'$. In short, we have clearly over-averaged the observations $Y_{i,n}$: having maximally reduced the influence of measurement errors, we have obscured the function $h$ so much that the estimation of $\mathrm{I}(h)$ has become impossible. Clearly, we need to adopt a more tempered approach.

Hence, we group the observations into only $M<n$ groups $G_{j,n}$, $j=1,\dots , M$, whose cardinalities $N:=\#(G_{j,n})$ we assume to be the same for all $j=1,\dots , M$. It is convenient to re-parametrize these choices using parameter $\alpha\in (0,1) $, which turns $M$ and $N$ into
\[
M=\lfloor n^{\alpha}\rfloor \quad \text{and} \quad N=\lfloor n^{1-\alpha}\rfloor .
\]
This re-parametrization is not artificial. It is, in a way, connected to smoothing histograms and estimating regression functions, and in particular to bandwidth selection in these research areas. We shall elaborate on this topic more in the next section. At the moment, we only note that the aforementioned connection plays a pivotal role in obtaining practically useful and sound estimates of the parameter $\alpha $.

To gain additional intuition on the grouping parameter $\alpha $, we come back for a moment to our numerical example with the one-day observation period, which is comprised of $n=86400$ observations, one per second. Suppose that we decide to average the sixty observations within each minute. Thus, we have $N=60$ and in this way produce $M=1440$ new data points, which we denote by $\widetilde{Y}_{j,n}$. Since $NM=n$, we have $\alpha=1-\log(N)/\log(n)$ and thus $\alpha =0.6398$. If, however, instead of averaging minute-worth data we decide to average, for example, hour-worth data, then we have $N=360$ (=group cardinality), $M=240$ (=number of groups), and thus $\alpha=0.4822$.

Continuing our general discussion, we average the original observations $Y_{i,n}$, $i=1\dots , n$, falling into each group $G_{j,n}$ and in this way obtain $M$ group-averages
\[
\widetilde{Y}_{j,n}:={1\over N}\sum_{i\in G_{j,n}} Y_{i,n}, \quad  j=1,\dots, M.
\]
Based on these averages, we modify the earlier introduced index $\mathrm{I}_n(h,\varepsilon)$ as follows:
\begin{equation}\label{approx-random-grouped}
\widetilde{\mathrm{I}}_{n,\alpha}(h,\varepsilon):={\sum_{j=2}^M (\widetilde{Y}_{j,n}-\widetilde{Y}_{j-1,n})_{+} \over \sum_{j=2}^M |\widetilde{Y}_{j,n}-\widetilde{Y}_{j-1,n}| }.
\end{equation}
The problem that we now face is to find, if exist, those values of $\alpha\in (0,1)$ that make the index $\widetilde{\mathrm{I}}_{n,\alpha}(h,\varepsilon)$ converge to $ \mathrm{I}(h)$ when $n\to \infty $. This is the topic of the next section.

\section{Consistency}
\label{consistency}

The following theorem establishes consistency of the estimator $\widetilde{\mathrm{I}}_{n,\alpha}(h,\varepsilon)$ and, in particular, specifies the range of possible $\alpha$ values.

\begin{theorem}
\label{th-1}
Let $h$ be a differentiable function defined on the unit interval $[0,1]$, and let its derivative $h'$ be $\gamma$-H\"{o}lder continuous for some $\gamma\in (0,1]$. If $\alpha \in (0, 1/3)$, then  $\widetilde{\mathrm{I}}_{n,\alpha}(h,\varepsilon)$ is a consistent estimator of $ \mathrm{I}(h)$, that is, when $n\to \infty$, we have
\begin{equation}\label{consistency-0}
\widetilde{\mathrm{I}}_{n,\alpha}(h,\varepsilon)\stackrel{\mathbf{P}}{\to} \mathrm{I}(h).
\end{equation}
The rate of convergence is of the order
\begin{equation}\label{consistency-1}
O_{\mathbf{P}}(1) n^{-\min\{\delta(\alpha ), \rho(\alpha ) \}}
\end{equation}
with  $\delta(\alpha )=\alpha \gamma$ arising from the deterministic part of the problem, that is, associated with the function $h$, and $\rho(\alpha )= (1-3\alpha)/2 $ arising from the random part, that is, associated with the measurement errors $\varepsilon_i$'s.
\end{theorem}

We next discuss the choice of $\alpha $ from the theoretical and practical perspectives, which do not coincide due to a number of reasons, such as the fact that theory is concerned with asymptotics when $n\to \infty $, while practice deals with finite values of $n$, though possibly very large. Under the (practical) non-asymptotic framework, any value of $\alpha \in (0,1]$ is, in principle, acceptable because the quantities $O_{\mathbf{P}}(1)$ and $ n^{-\min\{\delta(\alpha ), \rho(\alpha ) \}}$ in the specification of convergence rate (\ref{consistency-1}) interact, as both of them depend on $h$ and $\alpha $.

Under the (theoretical) asymptotic framework, the values $\alpha=0$ and $1$ have to be discarded immediately, as we have already noted. The remaining $\alpha$'s should, as Theorem \ref{th-1} tells us, be further restricted to only those below $1/3$. Since we wish to chose $\alpha $ that results in the fastest rate of convergence, we maximize the function $\alpha \mapsto \min\{\delta(\alpha ), \rho(\alpha ) \}$ and get
\begin{equation}\label{max-alpha}
\alpha_{\max}={1\over 3+2\gamma }.
\end{equation}
For example, if the second derivative $h''(t)$ is uniformly bounded on the interval $[0,1]$, which is the case in all our illustrative examples, then $\gamma =1$ and thus $\alpha_{\max}=1/5$.

The grouping and averaging technique that we employ is closely related to smoothing in non-parametric density and regression estimation (e.g., Silverman, 1986; H\"{a}rdle, 1991; Scott, 2015; and references therein). To elaborate on this connection, we recall that the number of groups is $M\approx n^{\alpha }$, whose reciprocal
\begin{equation}\label{band=h}
b:=1/M\approx n^{-\alpha }
\end{equation}
would play the role of `bandwidth.' In non-parametric density and regression estimation, the optimal bandwidth is of the order $O( n^{-1/5})$ when $n\to \infty $, which in our case corresponds to $\alpha_{\max}=1/5$. Hence, $\alpha =0$ means only one bin/group and thus over-smoothing, whereas $\alpha =1$ means as many bins/groups as there are observations, and thus under-smoothing. Of course, as we have already noted above, the values $\alpha=0$ and $\alpha=1$ are excluded, unless all the measurement errors vanish, in which case smoothing is not necessary and thus $\alpha =1$ can be used, as we indeed did earlier when dealing with the numerical index $\mathrm{I}_n(h)$.

Thinking of the role of $\gamma$-H\"{o}lder continuity of $h'$ on the problem, it is useful to look at two extreme cases: First, when $\gamma =1$, we have $\alpha_{\max}=1/5$ from formula (\ref{max-alpha}), which corresponds (under weak conditions) to the optimal bandwidth $O( n^{-1/5})$ in non-parametric density and regression estimation. Second, when no smoothing is applied, like in the case of the histogram density-estimator, then  (under weak conditions) the optimal bandwidth is of the order $O( n^{-1/3})$, which corresponds to $\alpha_{\max}=1/3$ when $\gamma =0$, which essentially means boundedness but no continuity of $h'$.

Hence, choosing an appropriate value of the grouping parameter $\alpha $ is a delicate task. We next discuss two approaches: The first one is data-exploratory (visual) when we assume that we know the population and want to gain insights into what might happen in practice. The second, practice-oriented approach relies on the idea of cross-validation (e.g., Arlot and Celisse, 2010, Celisse, 2008; and references therein) and is designed to produce estimates of $\alpha $ based purely on data.

\subsection{Data exploratory (visual) choice of $\alpha $}

To gain intuition on how to estimate the grouping parameter $\alpha $ from data, we start out with the functions in quartet (\ref{quartet-1}), which we view as populations, and then we contaminate their observations with i.i.d.~errors $\varepsilon_i \sim \mathcal{N}(0,1)$ according to formula (\ref{approx-random-e}).

We have visualized the values of the estimator $\widetilde{\mathrm{I}}_{n,\alpha}(h,\varepsilon)$ with respect to $n$ and $\alpha$ in Figure~\ref{h1h2h3h4-group},
\begin{figure}[h!]
  \centering
  \subfigure[The hyperplane at the height $\mathrm{I}(h_{1})=0.6667$]{%
    \includegraphics[width=0.5\textwidth]{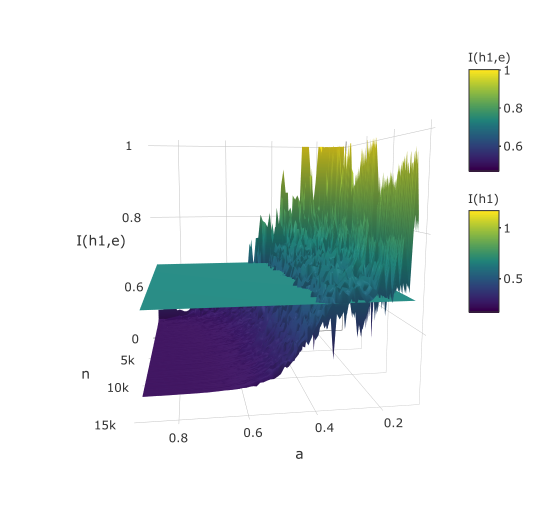}}%
\subfigure[The hyperplane at the height $\mathrm{I}(h_{2})=0.3333$]{%
    \includegraphics[width=0.5\textwidth]{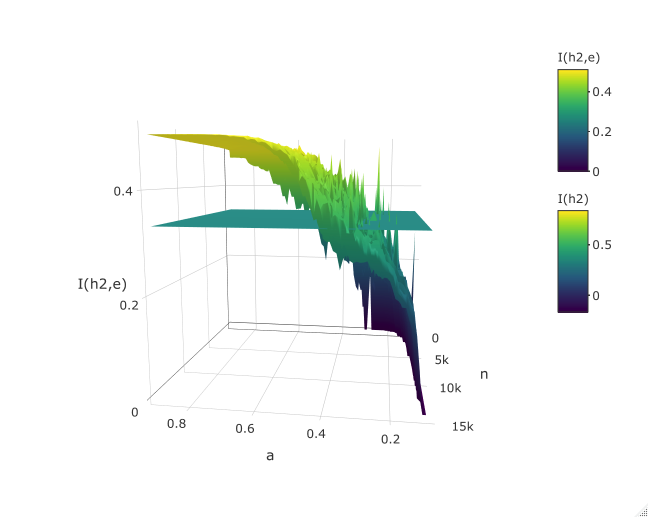}}%
    \\
 \subfigure[The hyperplane at the height $\mathrm{I}(h_{3})=1$]{%
    \includegraphics[width=0.5\textwidth]{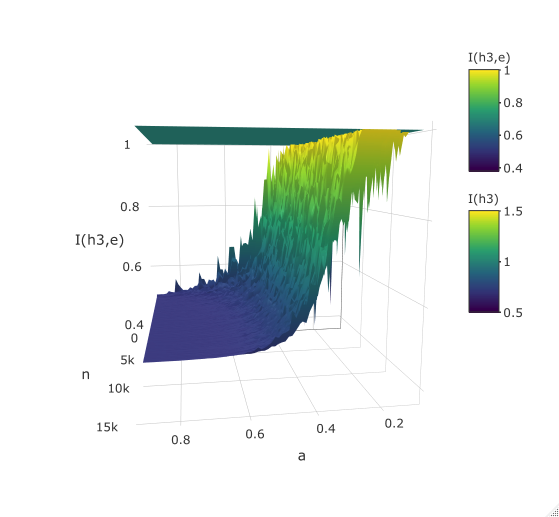}}%
\subfigure[The hyperplane at the height $\mathrm{I}(h_{4})=0$]{%
    \includegraphics[width=0.5\textwidth]{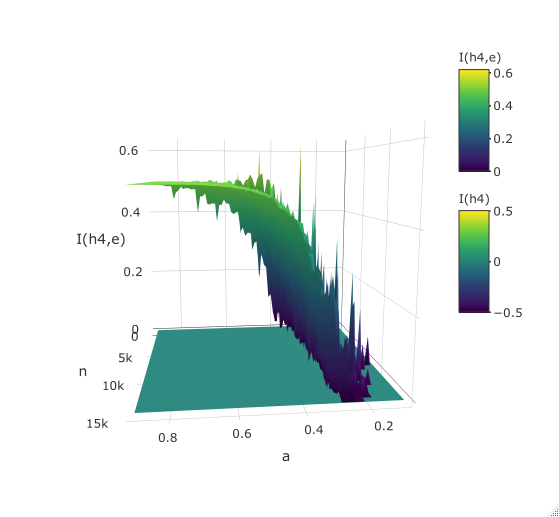}}%
    \caption{Values of $\widetilde{\mathrm{I}}_{n,\alpha}(h,\varepsilon)$ with respect to $n$ and $\alpha$ in the case of quartet (\ref{quartet-1}).}%
\label{h1h2h3h4-group}
\end{figure}	
where the hyperplane in each panel is at the height of the corresponding actual index of increase $\mathrm{I}(h)$. For each panel, we visually choose a value of $\alpha$ which is in the intersection of the curved surface with the hyperplane, because in this case the index $\widetilde{\mathrm{I}}_{n,\alpha}(h,\varepsilon)$ is close to the actual index $\mathrm{I}(h)$.

Even though the chosen parameter $\alpha$ value, which we denote by $\alpha_{\text{vi}}$, may not be optimal due to roughness of the surface, it nevertheless offers a sound choice, as we see from Figure~\ref{h1h2h3h4-group-fixed}
\begin{figure}[h!]
  \centering
  \subfigure[$\mathrm{I}(h_{1})=0.6667$, $\alpha_{\text{vi}}=0.35$]{%
    \includegraphics[height=0.35\textwidth]{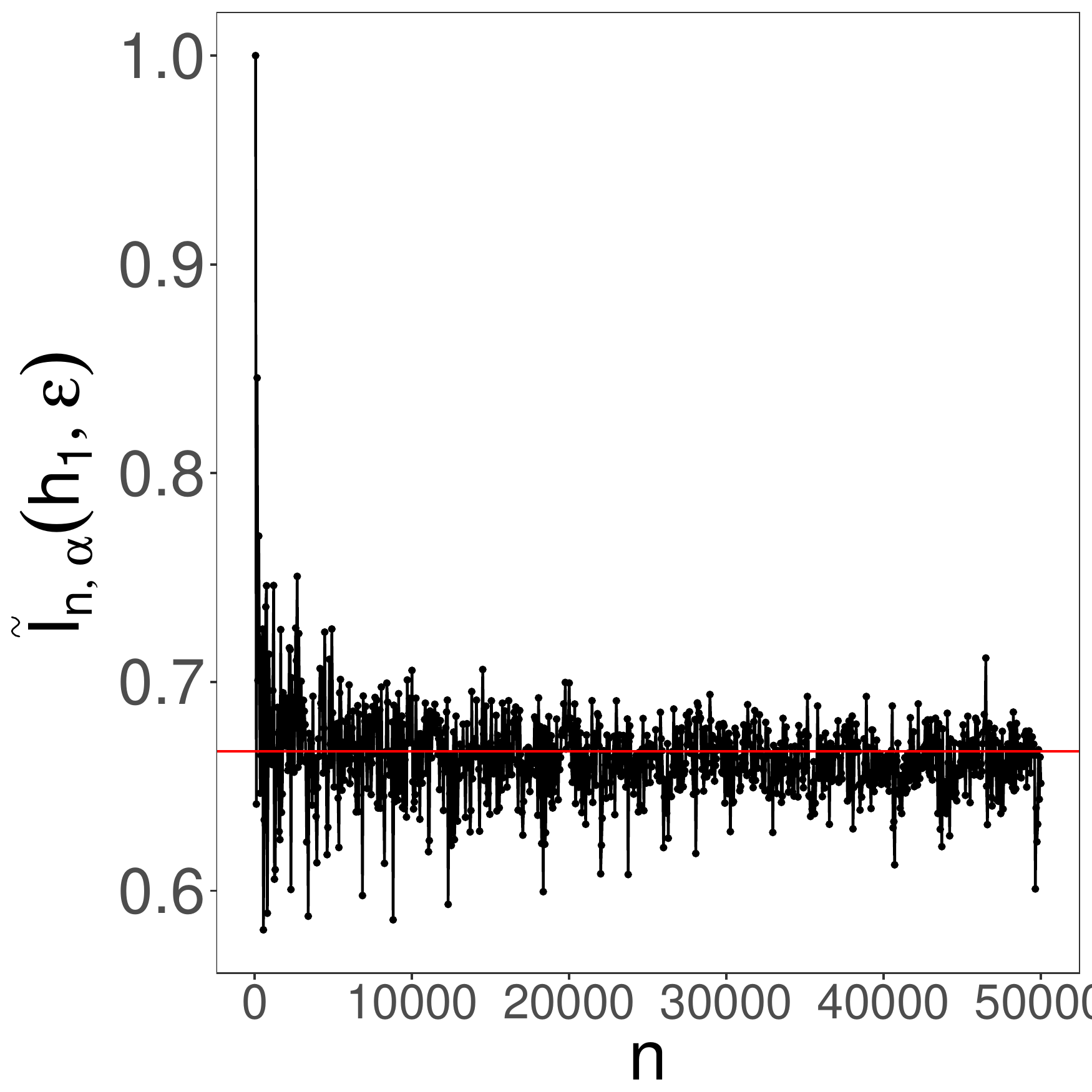}}%
\hspace{0.5in}
\subfigure[$\mathrm{I}(h_{2})=0.3333$, $\alpha_{\text{vi}}=0.33$]{%
    \includegraphics[height=0.35\textwidth]{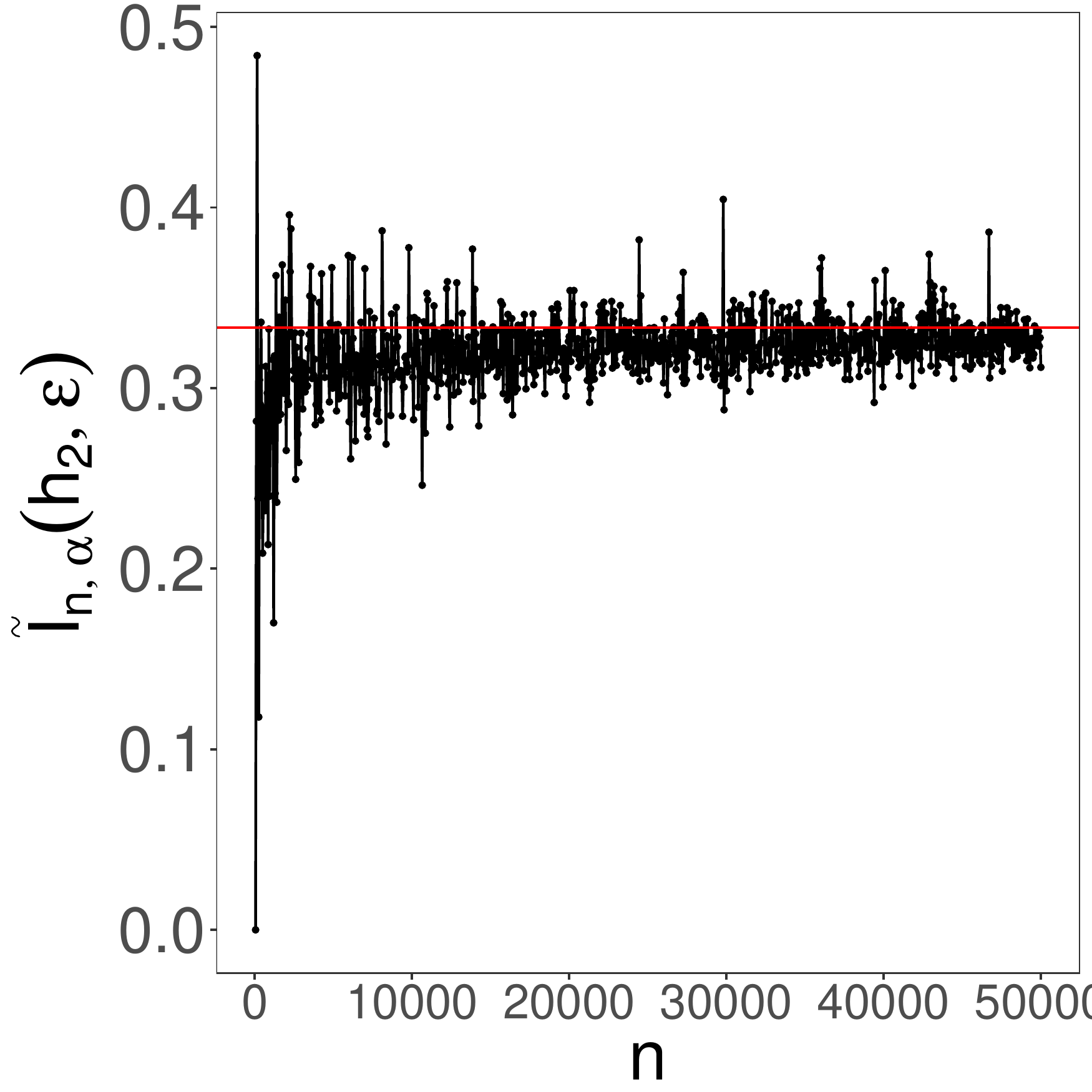}}%
\hspace{0.5in}
 \subfigure[$\mathrm{I}(h_{3})=1$, $\alpha_{\text{vi}}=0.1$]{%
    \includegraphics[height=0.35\textwidth]{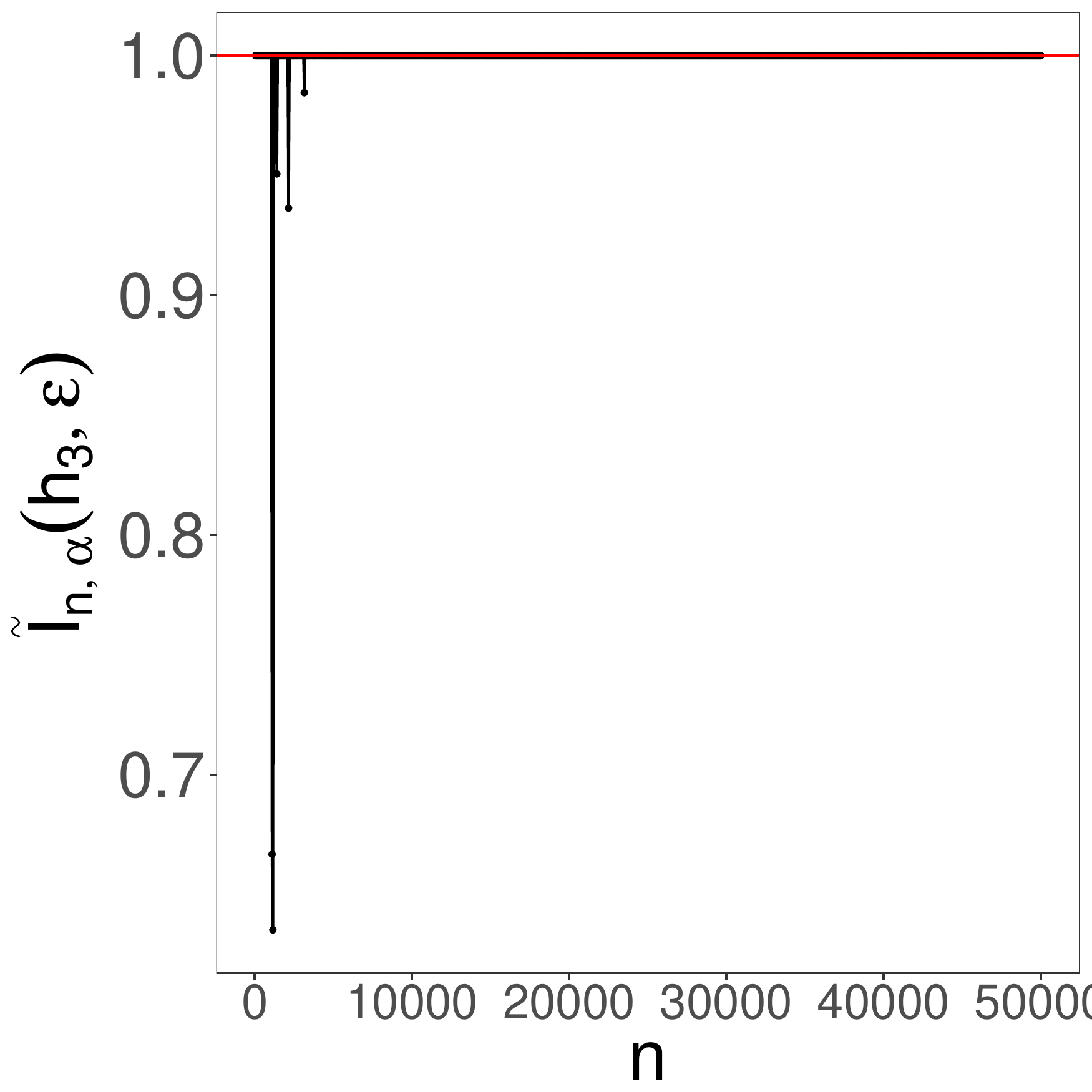}}%
\hspace{0.5in}
\subfigure[$\mathrm{I}(h_{4})=0$, $\alpha_{\text{vi}}=0.1$]{%
    \includegraphics[height=0.35\textwidth]{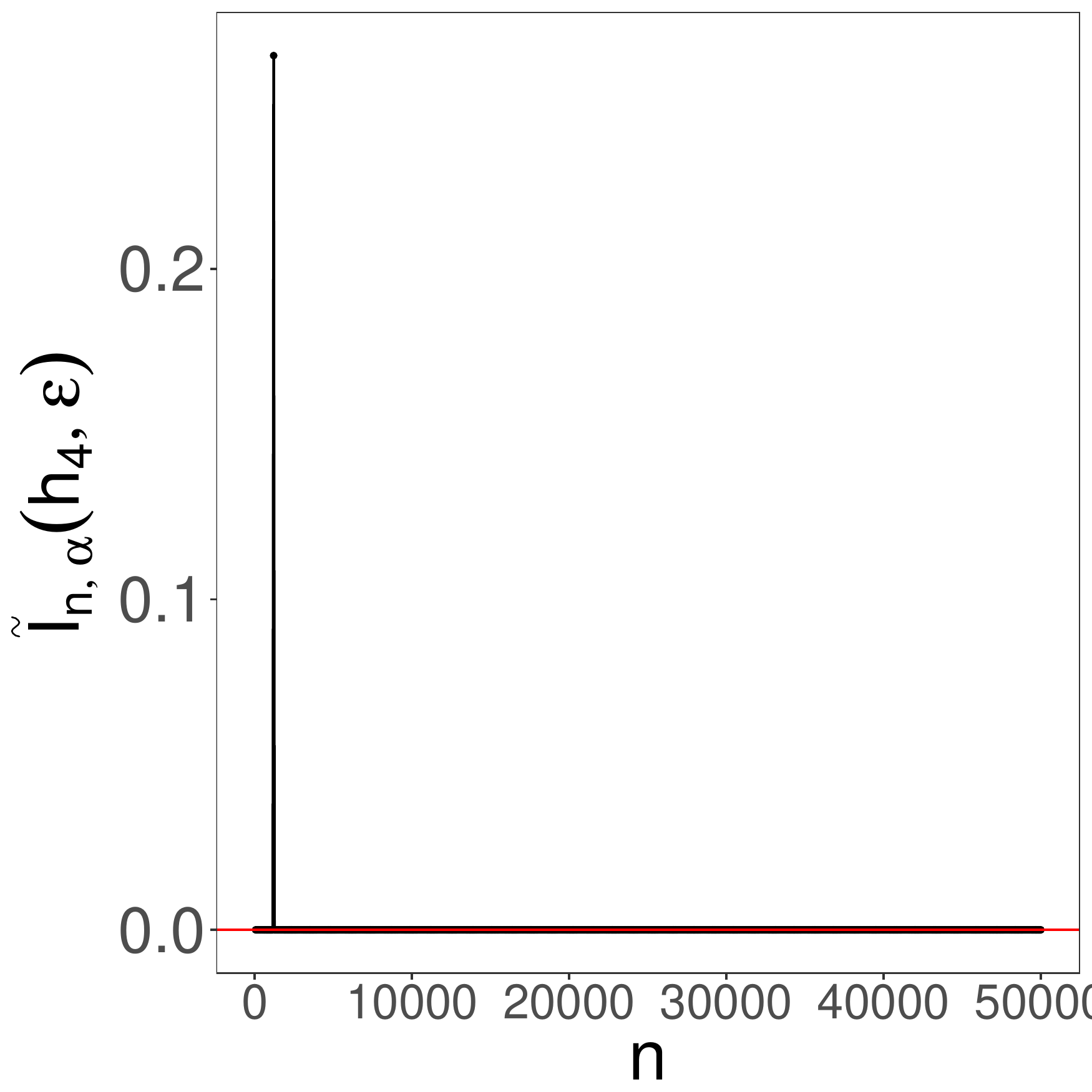}}%
    \caption{The performance of $\widetilde{\mathrm{I}}_{n,\alpha}(h,\varepsilon)$ with respect to $n$ in the case of quartet (\ref{quartet-1}) and based on vidual $\alpha$'s.}%
\label{h1h2h3h4-group-fixed}
\end{figure}	
where we depict the convergence of $\widetilde{\mathrm{I}}_{n,\alpha}(h,\varepsilon)$ to $\mathrm{I}(h)$ when $n$ grows. In each panel, the horizontal red `reference' line is at the height of the actual index value.

Note that in panel (a) of Figure~\ref{h1h2h3h4-group-fixed}, the visually obtained $\alpha_{\text{vi}}=0.35$ is slightly larger than $1/3$, but we have to say that we had decided on this value (as a good estimate) before we knew the result of Theorem \ref{th-1}, and thus before we knew the (theoretical) restriction $\alpha<1/3$. Nevertheless, we have decided to leave the value $\alpha_{\text{vi}}=0.35$ as it is, without tempering with our initial guess in any way. As we shall see in next Section \ref{cv-42}, however, the purely data-driven and based on cross-validation $\alpha $ value is $\alpha_{\text{cv}}=0.28$, which is within the range $(0,1/3)$ of theoretically acceptable $\alpha $ values.

\subsection{Choosing $\alpha $ based on cross validation}
\label{cv-42}

As we have already elucidated, equation (\ref{band=h}) connects our present problem with nonparametric regression-function estimation. In the latter area, researchers usually choose the optimal bandwidth as the point at which cross-validation scores become minimal (e.g., Arlot and Celisse, 2010, Celisse, 2008; and references therein). We adopt this viewpoint as well. Namely, given a scatterplot, say $(t_{i,n},Y_{i,n})$, we cross validate it (computational details and R packages will be described in a moment). Then we find the minimizing value $b=b_{\text{cv}}$ and finally, according to equation (\ref{band=h}),  arrive at the `optimal' $\alpha_{\text{cv}}$ via the equation
\begin{equation}\label{cv-eq}
\alpha_{\text{cv}}=\log(1/b_{\text{cv}})/\log(n).
\end{equation}
In Figure \ref{h1h2h3h4-cv},
\begin{figure}[h!]
  \centering
  \subfigure[Function $h_1$]{%
    \includegraphics[height=0.35\textwidth]{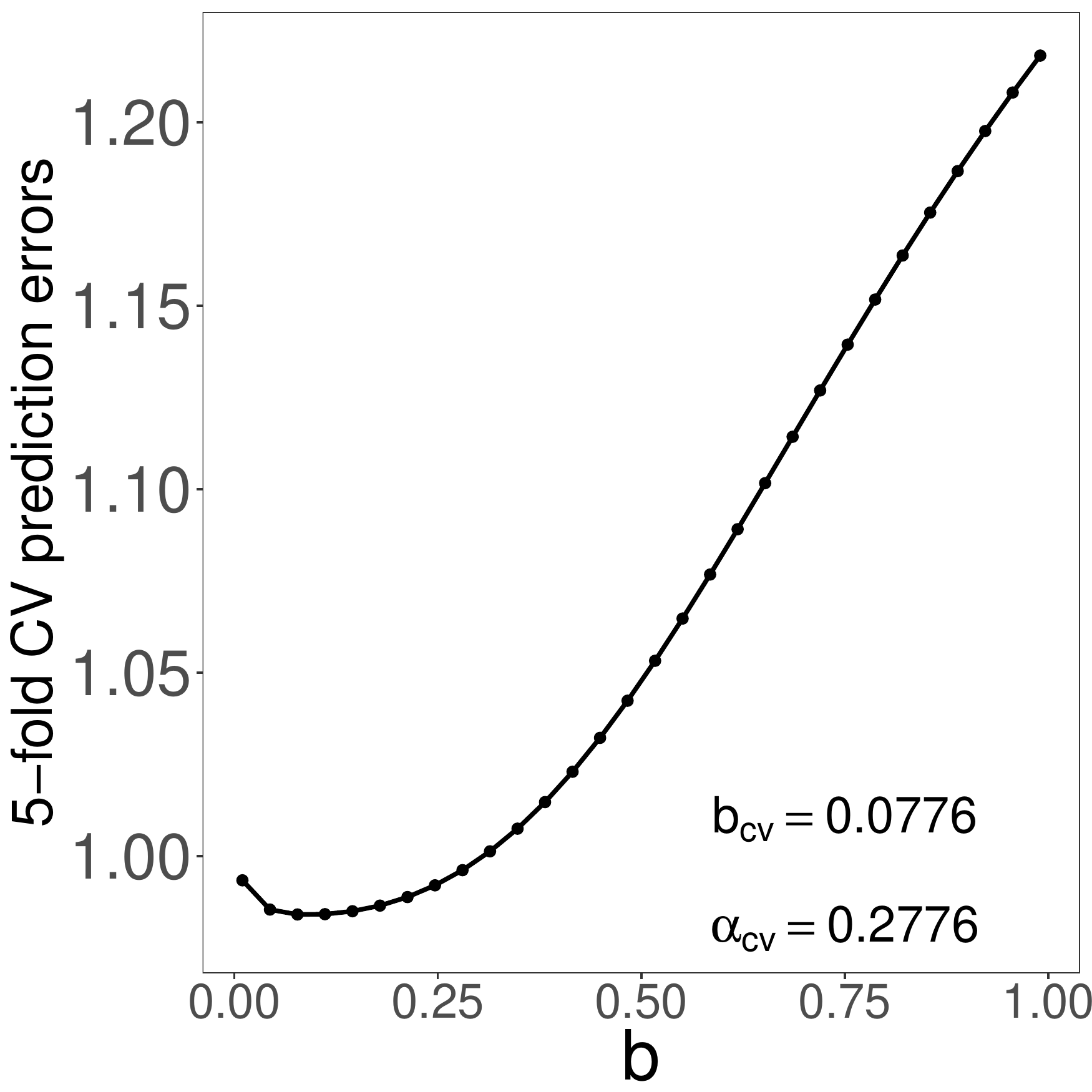}}%
\hspace{0.5in}
\subfigure[Function $h_2$]{%
    \includegraphics[height=0.35\textwidth]{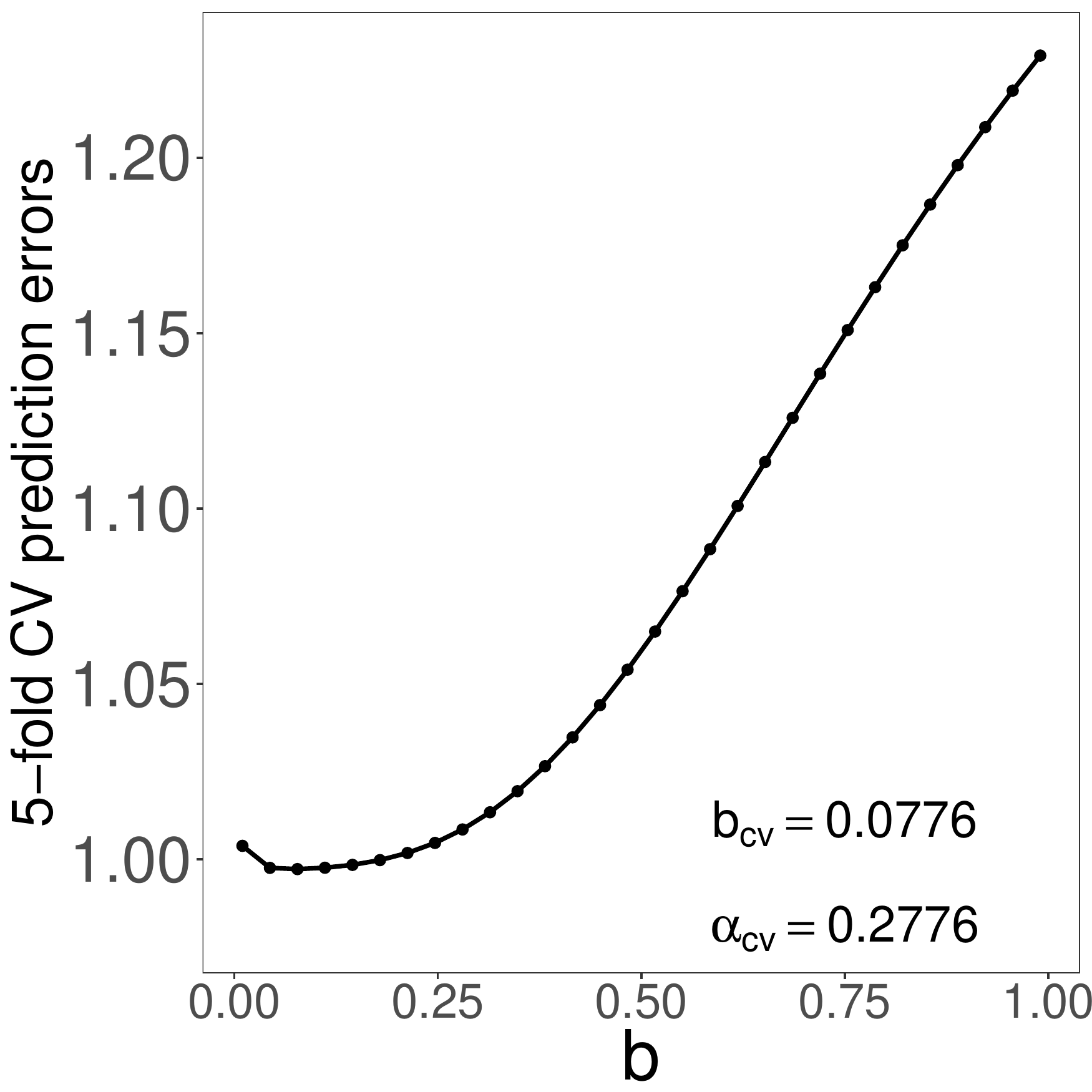}}%
\hspace{0.5in}
 \subfigure[Function $h_3$]{%
    \includegraphics[height=0.35\textwidth]{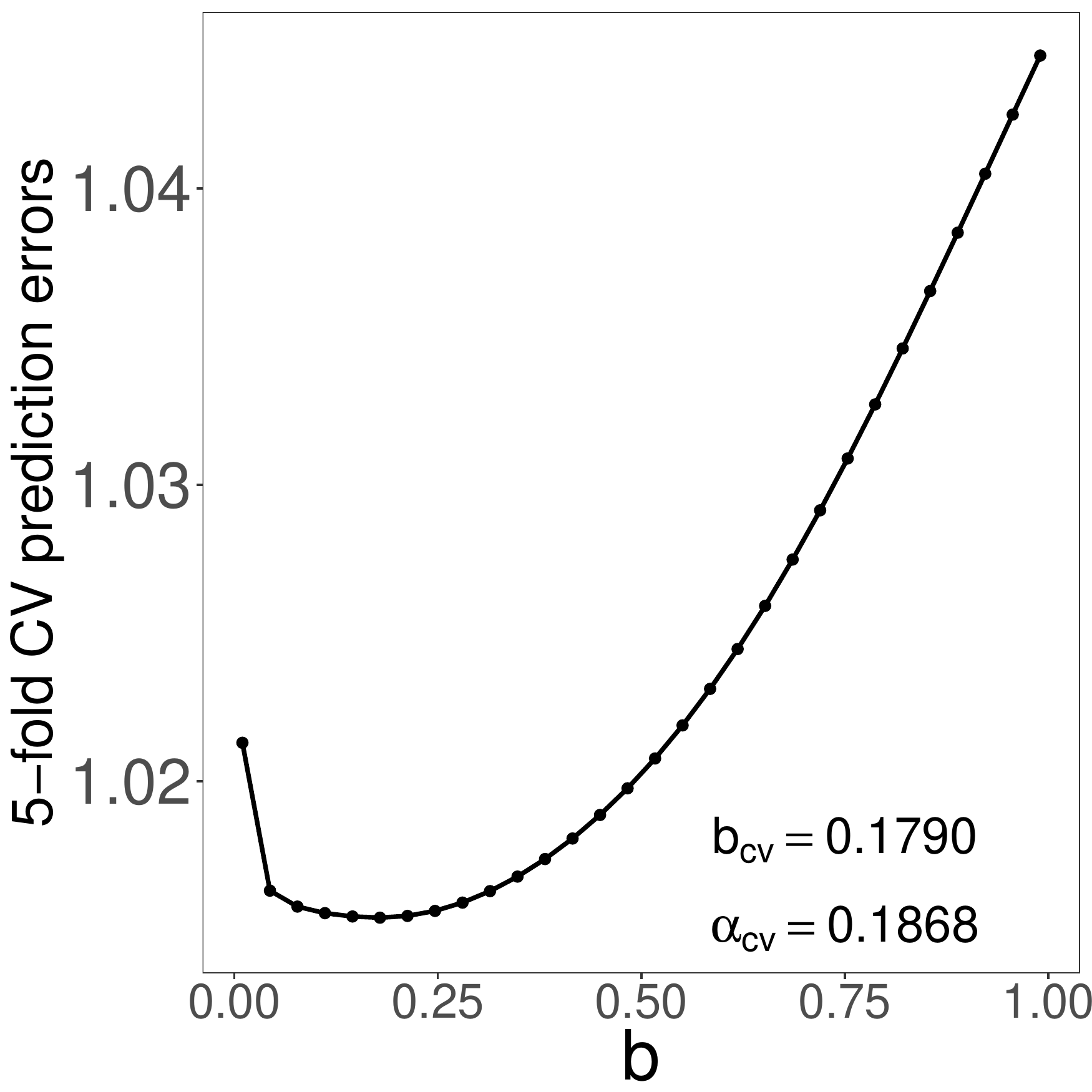}}%
\hspace{0.5in}
\subfigure[Function $h_4$]{%
    \includegraphics[height=0.35\textwidth]{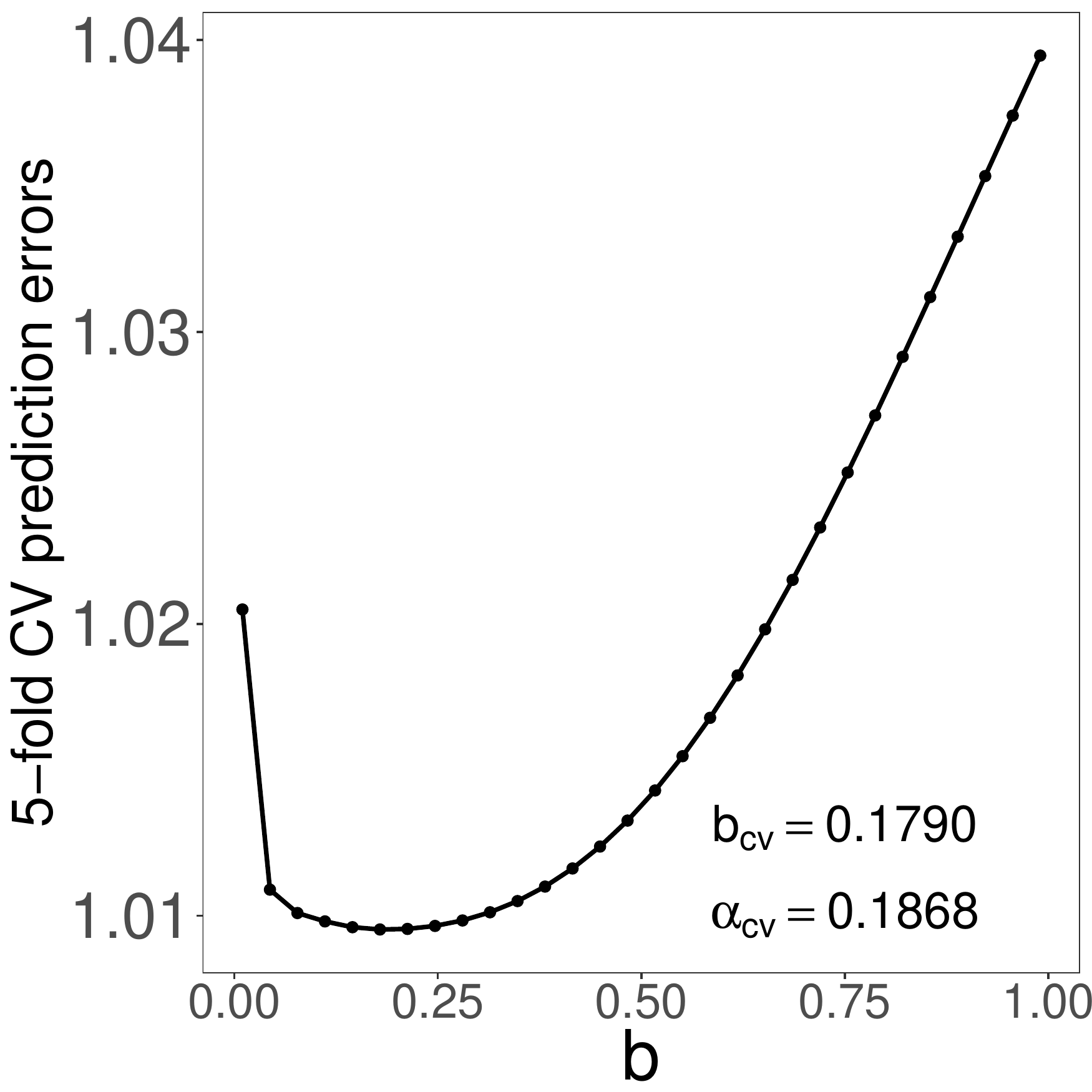}}%
    \caption{Cross validation, minima $b_{\text{cv}}$, and the grouping parameters  $\alpha_{\text{cv}}$ for quartet (\ref{quartet-1}).}%
\label{h1h2h3h4-cv}
\end{figure}	
we see some differences between the values of $\alpha_{\text{vi}}$ and $\alpha_{\text{cv}}$. Nevertheless, we should not prejudge the situation in any way because in practice, when no hyperplanes can be produced due to unknown values of $\mathrm{I}(h)$, only the values of $\alpha_{\text{cv}}$ can be extracted from data. Note, however, that the four values of $\alpha_{\text{cv}}$ reported in the panels of Figure \ref{h1h2h3h4-cv} are in compliance with the condition of Theorem \ref{th-1} stipulating that $\alpha$'s must be in the range $(0,1/3)$ in order to have (asymptotic) consistency.

To explore how the grouped estimator $\widetilde{\mathrm{I}}_{n,\alpha}(h,\varepsilon)$ based on $\alpha_{\text{cv}}$'s actually performs, we have produced Figure \ref{h1h2h3h4-group-fixed-cv}.
\begin{figure}[h!]
  \centering
  \subfigure[$\mathrm{I}(h_{1})=0.6667$, $\alpha_{\text{cv}}=0.28$]{%
    \includegraphics[height=0.35\textwidth]{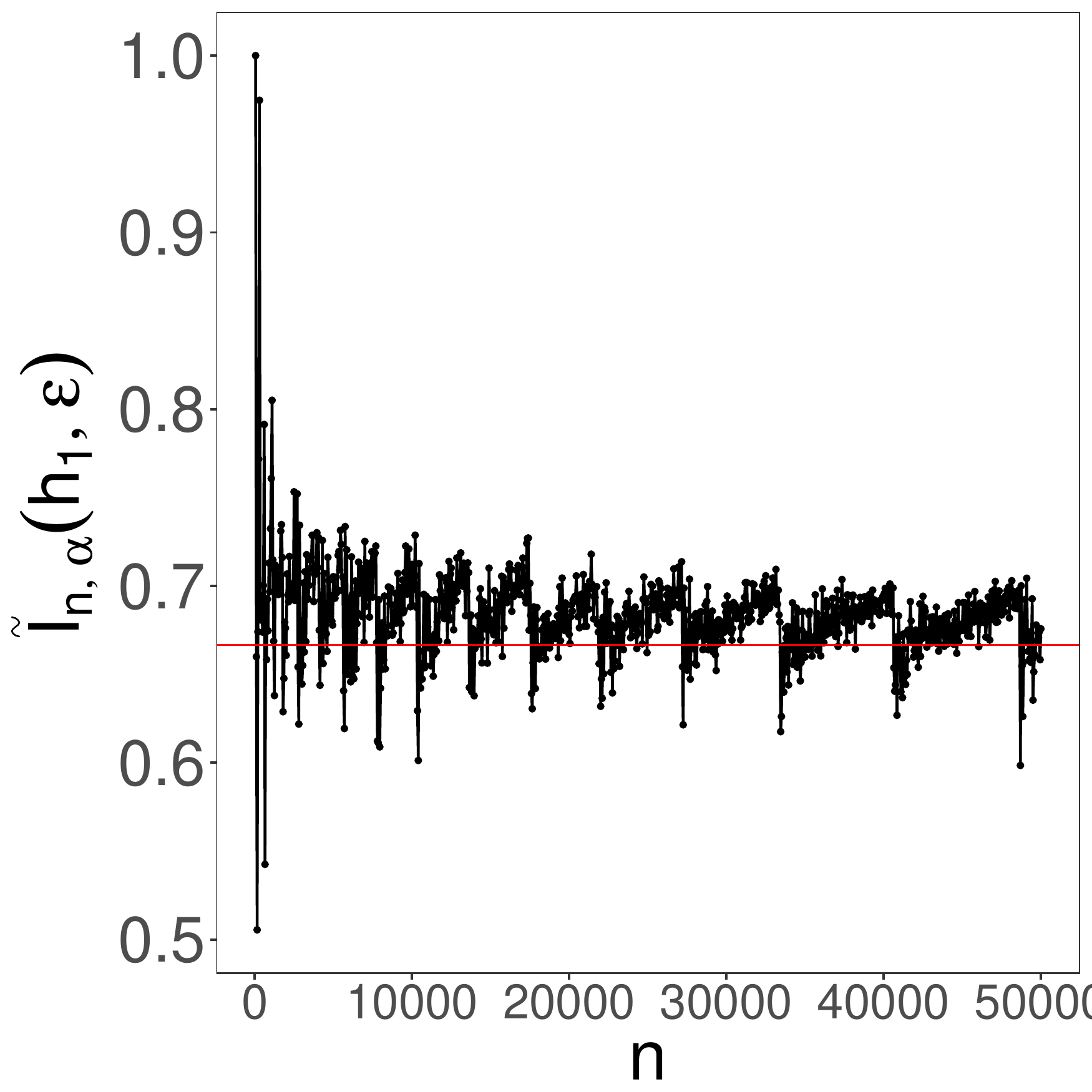}}%
\hspace{0.5in}
\subfigure[$\mathrm{I}(h_{2})=0.3333$, $\alpha_{\text{cv}}=0.28$]{%
    \includegraphics[height=0.35\textwidth]{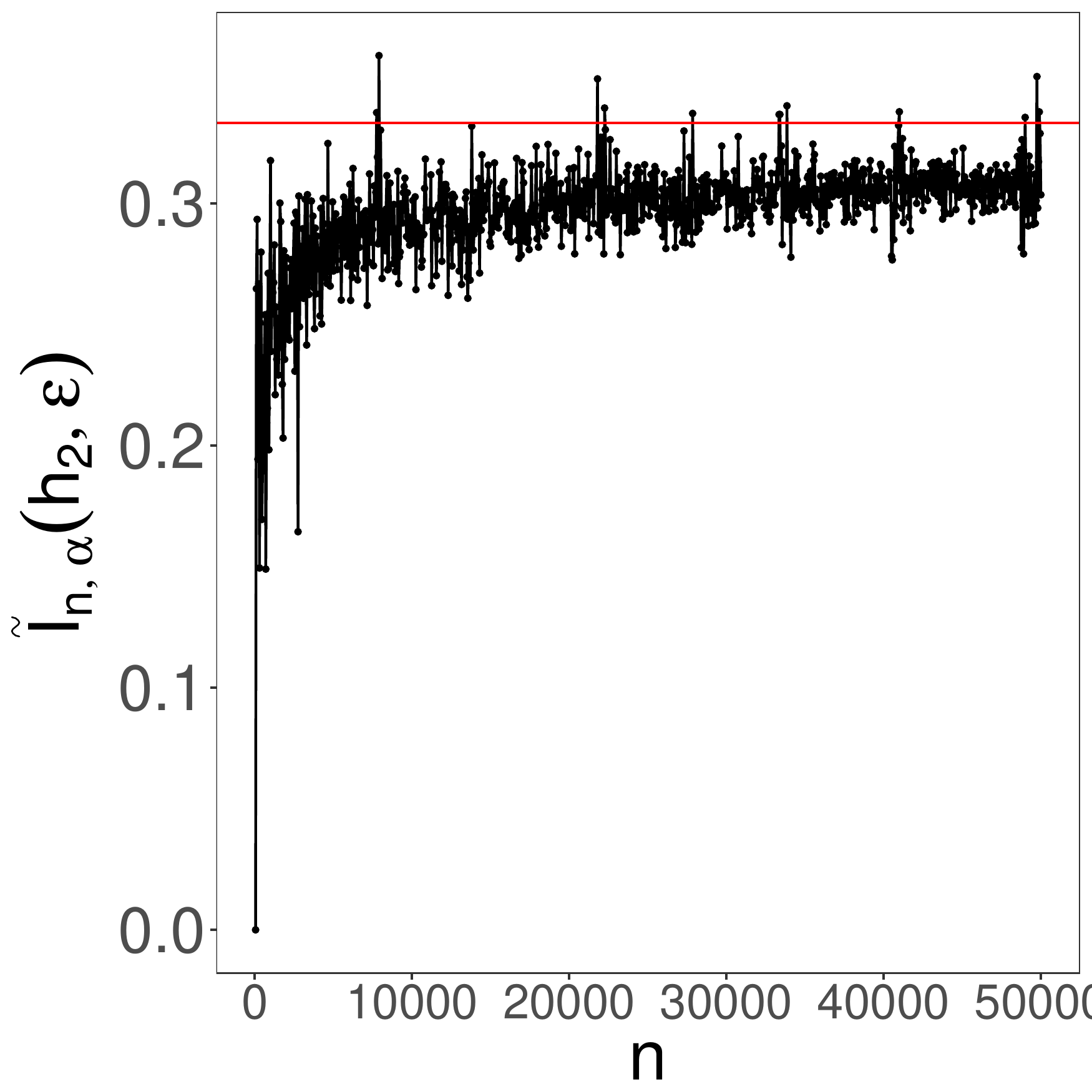}}%
\hspace{0.5in}
 \subfigure[$\mathrm{I}(h_{3})=1$, $\alpha_{\text{cv}}=0.19$]{%
    \includegraphics[height=0.35\textwidth]{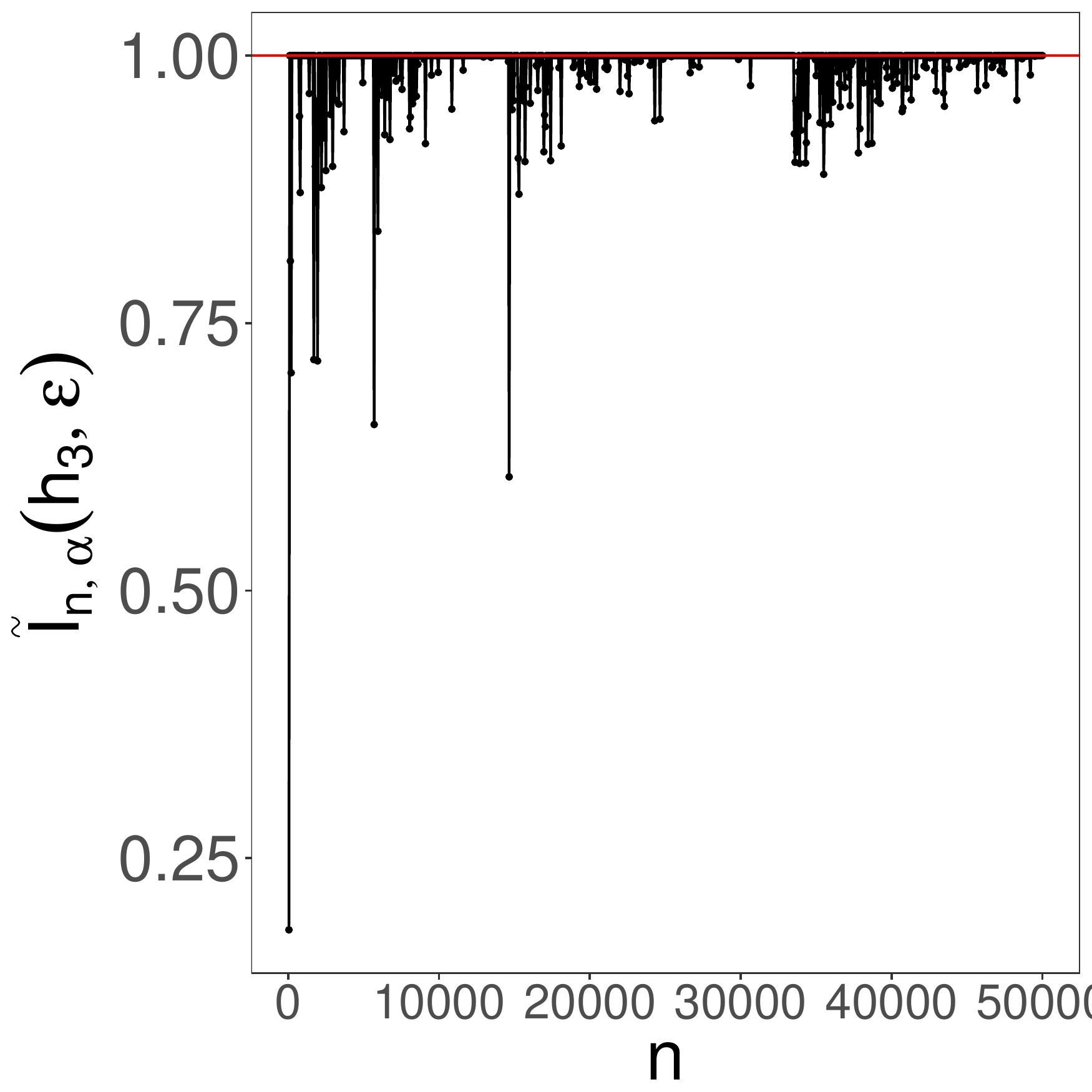}}%
\hspace{0.5in}
\subfigure[$\mathrm{I}(h_{4})=0$, $\alpha_{\text{cv}}=0.19$]{%
    \includegraphics[height=0.35\textwidth]{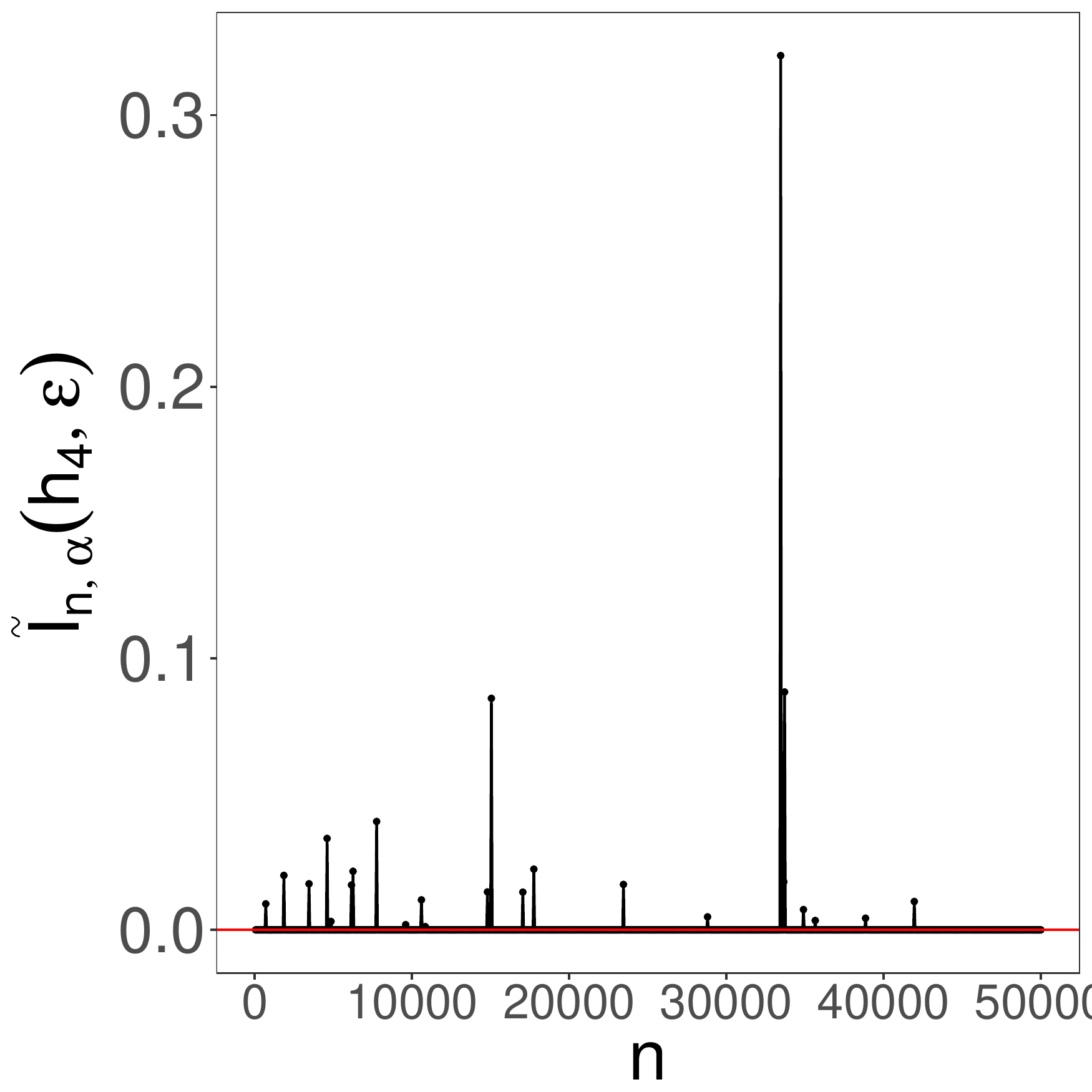}}%
    \caption{The performance of $\widetilde{\mathrm{I}}_{n,\alpha}(h,\varepsilon)$ with respect to $n$ in the case of quartet (\ref{quartet-1}) and cross validation.}%
\label{h1h2h3h4-group-fixed-cv}
\end{figure}	
Naturally, since the respective visual $\alpha_{\text{vi}}$'s and cross-validatory  $\alpha_{\text{cv}}$'s do not coincide, the corresponding values of $\widetilde{\mathrm{I}}_{n,\alpha}(h,\varepsilon)$ are also different. Which of them are better from the statistical point of view will become clearer only in Section \ref{bootstrap}, where bootstrap-based standard errors and confidence intervals are derived.

We next present a detailed implementation procedure for finding cross-validatory estimates $\alpha_{\text{cv}}$ of the grouping parameter $\alpha $. Naturally, the help of the R computing language (R Core Team, 2013) becomes indispensable, and we have used a number of R packages to accomplish the task. We also wish to acknowledge the packages \verb"ggplot2" (Wickham, 2009) and \verb"plotly" (Sievert et al., 2017) that we have used extensively in this paper to draw two-dimentional plots and interactive surface plots; the latter plots have been pivotal in extracting the values $\alpha_{\text{vi}}$  visually.

Hence, from the purely practical computational perspective, we now utilize bandwidth selection techniques of kernel-based regression-function estimation in order to get estimates of the grouping parameter. First, for the sake of programming efficiency, we restrict $b$'s to the interval $ (0.01, 0.99)$, and we evenly split the latter interval into bins of width $(0.99-0.01)/29 \approx 0.0338$, all of which can of course be refined in order to achieve, if desired, smaller computational errors. Hence, from now on, we have thirty equidistant $b$'s, which are $b_i\approx 0.01+(i-1)0.0338$ for $i=1,\dots , 30$. Next we use the common cross-validation method called repeated $k$-fold cross validation, and we set $k=5$ for our purpose. The following main steps are:
\begin{enumerate}
    \item\label{s1}
    For each function $h$ under consideration, we generate $n=10000$ data points based on equation (\ref{approx-random-e}).
	\item\label{s2}
We randomly split the given $n$ points into $k$ folds, denoted by $D_{1},\dots , D_{k}$,  of roughly equal sizes.
	\item\label{s3}
For each value $b_i$, we use $D_{1}$ as the validation set and let other $D$'s be training sets, which we use to fit a kernel regression model. Specifically, we use the function \verb"ksmooth" from the R package \verb"stats", with the parameter \verb"kernel" set to \verb"normal", which means that we use the normal kernel. Then we use the validation set $D_{1}$ to get the predicted values and calculate one prediction error, defined as the mean-square error and denoted by $E_{1}$. We repeat this step until we use up all the folds as our validation sets. Hence, we obtain $k$ prediction errors $E_{1},\dots , E_{k}$. Finally, we average these $k$ prediction errors and denote this average by  $E_{b_{i},1}$.
	\item\label{s4}
We repeat Step \ref{s3} for all $b_i$'s, thus arriving at one estimated prediction error for each $b_i$. Hence, in total, we have $E_{b_{1},1}, \dots, E_{b_{30},1}$.
	\item\label{s5}
We repeat Steps \ref{s1}--\ref{s4} fifty times, for every $b_i$, and then take the averages of the corresponding fifty estimated prediction errors. This gives us fifty final estimates, which we denote by $E_{b_{i}}$. For example, for $b_{1}$, the final estimate $E_{b_{1}}$ is the average of $E_{b_{1},1}, \dots, E_{b_{1},50}$. In summary, after this step, we have $E_{b_{1}}, \dots , E_{b_{30}}$ of the final estimates of the prediction error.
	\item\label{s6}
We draw the plot of the $b_i$'s versus the corresponding estimated prediction errors. The $b_i$ that gives the minimal prediction error is denoted by $b_{\text{cv}}$. Finally, we use equation (\ref{cv-eq}) to get $\alpha_{\text{cv}}$.
\end{enumerate}

\subsection{Proof of Theorem \ref{th-1}}

The following lemma, whose special case is Proposition \ref{prop-a3} formulated earlier, plays a pivotal role when proving Theorem \ref{th-1}.

\begin{lemma}
\label{le-a2}
Let $h$ be differentiable, and let its derivative $h'$ be $\gamma$-H\"{o}lder continuous for some $\gamma\in (0,1]$. Furthermore, let $\ell$ be any positively homogeneous and Lipschitz function. Then there is a constant $c<\infty $ such that, for any set of points $s_1:=0<s_2<\cdots <s_M\le 1$,
\begin{equation}\label{partition-1}
\sum_{j=2}^M \ell \Big (h(s_{j})-h(s_{j-1})\Big )
= \int_0^{s_M} \ell \big ( h' \big )\text{d}\lambda
+ \theta c \sum_{j=2}^M  | s_{j}-s_{j-1} |^{1+\gamma }
\end{equation}
where $\theta $ is such that $|\theta|\le 1$.
\end{lemma}

\begin{proof}
Since $\ell$ is Lipschitz and $h'$ is $\gamma$-H\"{o}lder continuous, we have
\begin{align}
\int_0^{s_M} \ell \big ( h' \big )\text{d}\lambda
&= \sum_{j=2}^M \int_{s_{j-1}}^{s_{j}} \ell \big ( h'(s) \big )
-\ell \big ( h'(s_{j}) \big )\text{d}s
+ \sum_{j=2}^M \big (s_{j}-s_{j-1} \big ) \ell \big ( h'(s_{j}) \big )
\notag
\\
&= \theta c \sum_{j=2}^M \int_{s_{j-1}}^{s_{j}} | s- s_{j} |^{\gamma }\text{d}s
+ \sum_{j=2}^M \big (s_{j}-s_{j-1} \big ) \ell \big ( h'(s_{j}) \big )
\notag
\\
&= \theta c \sum_{j=2}^M  | s_{j} -s_{j-1}|^{1+\gamma }
+ \sum_{j=2}^M \big (s_{j}-s_{j-1} \big ) \ell \big ( h'(s_{j}) \big ),
\label{a2-1}
\end{align}
where the values of $c<\infty $ and $|\theta|\le 1$ might have changed from line to line. Next, we explore the right-most sum of equation (\ref{a2-1}), to which we add and subtract the right-hand side of equation (\ref{partition-1}). Then we use the mean-value theorem with some $\xi_j\in [s_{j-1},s_{j}]$ and arrive at the equations
\begin{align}
\sum_{j=2}^M \big (s_{j}-s_{j-1} \big ) \ell \big ( h'(s_{j}) \big )
&= \sum_{j=2}^M \ell \Big (h(s_{j})-h(s_{j-1})\Big )
+\sum_{j=2}^M (s_{j}-s_{j-1}) \Big ( \ell \big ( h'(s_{j}) \big )
-\ell \big ( h'(\xi_{j})\big ) \Big )
\notag
\\
&= \sum_{j=2}^M \ell \Big (h(s_{j})-h(s_{j-1})\Big )
+\theta c \sum_{j=2}^M  | s_{j} -s_{j-1}|^{1+\gamma },
\label{a2-2}
\end{align}
where the last equation holds because $\ell$ is positively homogeneous and Lipschitz, and $h'$ is $\gamma$-H\"{o}lder continuous. Equations (\ref{a2-1}) and (\ref{a2-2}) imply equation (\ref{partition-1}) and finish the proof of Lemma \ref{le-a2}.
\end{proof}

\begin{proof}[Proof of Theorem \ref{th-1}]
We start with the equations
\begin{align}
\widetilde{Y}_{j,n}&={1\over N}\sum_{i\in G_{j,n}}  h(t_{i,n})
+{1\over N}\sum_{i\in G_{j,n}}  \varepsilon_i
\notag
\\
&={1\over N}\sum_{i\in G_{j,n}}  h(t_{i,n})
+\varepsilon^*_{j,n} ,
\label{ytilde-1}
\end{align}
where
\[
\varepsilon^*_{j,n}={1\over N}\sum_{i\in G_{j,n}}  \varepsilon_i.
\]
We next tackle the deterministic sum on the right-hand side of equation (\ref{ytilde-1}), and start with the equation
\[
{1\over N}\sum_{i\in G_{j,n}}  h(t_{i,n})
= {n-1\over N}\sum_{i=1}^{N}  h\bigg ( { i-1\over n-1}+{ (j-1)N\over n-1} \bigg ){1\over n-1}
\]
because $G_{j,n}=(j-1)N+\{1, \dots , N\}$ for all $j=1,\dots , M$. Consequently,
\begin{align}
{1\over N}\sum_{i\in G_{j,n}}  h(t_{i,n})
&= {n-1\over N} \bigg ( \sum_{i=1}^{N}  h\bigg ( { i-1\over n-1}+{ (j-1)N\over n-1} \bigg ){1\over n-1} -\int_{(j-1)N/(n-1)}^{jN/(n-1)} h\text{d}\lambda  \bigg )
\notag
\\
&\qquad + {n-1\over N} \int_{(j-1)N/(n-1)}^{jN/(n-1)} h\text{d}\lambda
\notag
\\
&= {n-1\over N} \int_{(j-1)N/(n-1)}^{jN/(n-1)} h\text{d}\lambda  + O(n^{-1 }),
\label{ytilde-1new2}
\end{align}
where we used the fact that $h$ is Lipschitz. By the mean-value theorem, there is $t_{j,n}^*$ between $(j-1)N/(n-1)$ and $jN/(n-1)$ such that the right-hand side of equation (\ref{ytilde-1new2}) is equal to $h(t_{j,n}^*)+ O(n^{-1})$.
Consequently, with the notation
\[
Y^*_{j,n}:=h(t_{j,n}^*)+\varepsilon^*_{j,n},
\]
we have $\widetilde{Y}_{j,n}=Y^*_{j,n}+O(n^{-1})$ and thus the increments  $\widetilde{Y}_{j,n}-\widetilde{Y}_{j-1,n}$ are equal to $Y^*_{j,n}-Y^*_{j-1,n}+O(n^{-1})$.  This gives us the equations
\begin{align}
\sum_{j=2}^M \ell\Big (\widetilde{Y}_{j,n}-\widetilde{Y}_{j-1,n}\Big )
&= \sum_{j=2}^M \ell\Big (Y^*_{j,n}-Y^*_{j-1,n}\Big ) +O(n^{-(1-\alpha)})
\notag
\\
&= \sum_{j=2}^M \ell \Big (h(t_{j,n}^*)-h(t_{j-1,n}^*)\Big )
+ O\bigg ( \sum_{j=2}^M |\varepsilon^*_{j,n}-\varepsilon^*_{j-1,n}| \bigg )
+O(n^{-(1-\alpha)})
\label{ytilde-1new3}
\end{align}
because $|\ell(t)-\ell(s)|\le |t-s|$ for all real $t$ and $s$.
The random variables $\varepsilon^*_{j,n}$, $j=1,\dots, M$, are independent and identically distributed with the means $0$ and variances $\sigma^2/N$. Hence,
\begin{align*}
\mathbf{E}\bigg ( \sum_{j=2}^M |\varepsilon^*_{j,n}-\varepsilon^*_{j-1,n}| \bigg )
&\le cM \max_{j}\sqrt{ \mathbf{E}\big ( (\varepsilon^*_{j,n})^2 \big ) }
\\
&\le cM \max_{j}\sqrt{ \sigma^2/N }
\\
&=O\big( n^{-(1-3\alpha)/2}\big ),
\end{align*}
which implies
\begin{equation}
\sum_{j=2}^M |\varepsilon^*_{j,n}-\varepsilon^*_{j-1,n}| =O_{\mathbf{P}}\big(n^{-(1-3\alpha)/2}\big ).
\label{eps-bound}
\end{equation}
The right-hand side of equation (\ref{eps-bound}) converges to $0$ because $\alpha \in (0,1/3)$. In view of equations (\ref{ytilde-1new3}) and (\ref{eps-bound}), we have
\begin{equation}
\sum_{j=2}^M \ell\Big (\widetilde{Y}_{j,n}-\widetilde{Y}_{j-1,n}\Big )
= \sum_{j=2}^M \ell \Big (h(t_{j,n}^*)-h(t_{j-1,n}^*)\Big )
+ O_{\mathbf{P}}\big(n^{-(1-3\alpha)/2}\big )+O(n^{-(1-\alpha)}).
\label{ytilde-2a}
\end{equation}
Furthermore, by Lemma \ref{le-a2} we have
\begin{equation}
\sum_{j=2}^M \ell \Big (h(t_{j,n}^*)-h(t_{j-1,n}^*)\Big )
= \int_0^1 \ell \big ( h' \big )\text{d}\lambda  +O\big( n^{-\alpha \gamma }\big ).
\label{ytilde-2c}
\end{equation}
Combining equations (\ref{ytilde-2a}) and (\ref{ytilde-2c}), and using $\beta$ to denote $\min\{\alpha \gamma, (1-3\alpha)/2\}$, we have
\begin{align*}
\widetilde{\mathrm{I}}_{n,\alpha}(h,\varepsilon)
&={\int_0^1 (h')_{+}\text{d}\lambda +O_{\mathbf{P}}\big( n^{-\beta }\big ) \over \int_0^1 |h'|\text{d}\lambda +O_{\mathbf{P}}\big( n^{-\beta }\big ) }
\notag
\\
&\stackrel{\mathbf{P}}{\to} {\int_0^1 (h')_{+}\text{d}\lambda \over \int_0^1 |h'|\text{d}\lambda }=\mathrm{I}(h) .
\end{align*}
The rate of convergence is of the order $ O_{\mathbf{P}}(n^{-\beta})$. Theorem~\ref{th-1} is proved.
\end{proof}

\section{Bootstrap-based confidence intervals}
\label{bootstrap}

To construct confidence intervals for $\mathrm{I}(h)$ based on the estimator $\widetilde{\mathrm{I}}_{n,\alpha}(h,\varepsilon)$, we need to determine standard errors, which turns out to be a very complex task from the viewpoint of asymptotic theory. Hence, we employ bootstrap (e.g., Hall, 1992; Efron and Tibshirani, 1993; Shao and Tu, 1995; Davison and Hinkley, 1997; and references therein). The re-sampling size $m$ is quite often chosen to be equal to the actual sample size $n$, but in our case, we find it better to re-sample fewer than $n$ observations (i.e., $m < n$) and thus follow specialized to this topic literature by Bickel et al. (1997), Bickel and Sakov (2008), Gribkova and  Helmers (2007, 2011); see also references therein. Specifically, the steps that we take are:
\begin{itemize}
	\item For a given function $h$, we generate $n=10000$ values $y_{1}, ...., y_{n}$ according to the model $Y_{i} = h(t_{i,n})+ \varepsilon_{i}$, where $\varepsilon_{i}$ are i.i.d. standard normal.
    \item We re-sample $1000$ times and in this way obtain 1000 sub-samples of size $m$, which we choose to be $m \approx 2\sqrt{n}$ according to a rule of thumb (DasGupta, 2008, p.~478).
	\item We use formula (\ref{approx-random-grouped}) to calculate the grouped index of increase, thus obtaining 1000 values of it; one value for each sub-sample. We denote the empirical distribution of the obtained values by $F^{*}$.
	\item With $Q^{*}$ denoting the (generalized) inverse of $F^{*}$, the 95\% quantile-based confidence interval is $(q_{2.5\%}, q_{97.5\%})$, where $q_{2.5\%}= Q^{*}(0.025)$ and $q_{97.5\%}= Q^{*}(0.975)$.
\end{itemize}

To illustrate, we introduce a second quartet of functions of this paper, namely:
\begin{equation}
\begin{split}
h_{5}(t)=(t-1)^{2}+ \sin(6t),
&\quad h_{6}(t)=(t-0.25)^{2}+\sin(0.25t),
\\
h_{7}(t)=t^{3}-5.6t^{2}+6t,
&\quad h_{8}(t)=\sin(2\pi t).
\end{split}
\label{quartet-2}
\end{equation}
We have visualized the functions in Figure \ref{fig-quartet-2}.
\begin{figure}[h!]
\centering
  \centering
  \subfigure[$\mathrm{I}(h_{5})\approx \mathrm{I}_n(h_{5})=0.3311$]{%
    \includegraphics[height=0.3\textwidth]{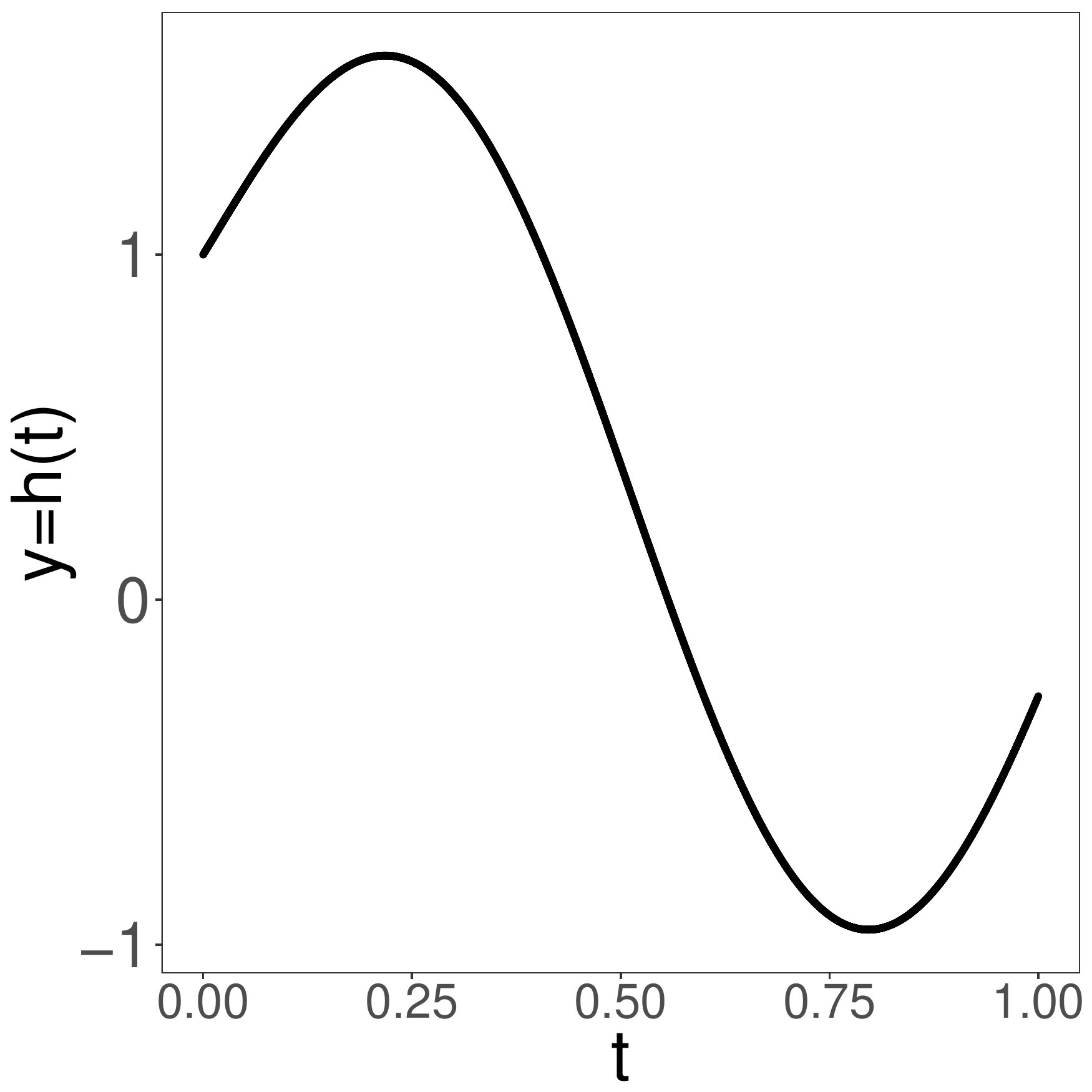}}%
\hspace{0.5in}
\subfigure[$\mathrm{I}(h_{6})\approx\mathrm{I}_n(h_{6})=0.9799$]{%
    \includegraphics[height=0.3\textwidth]{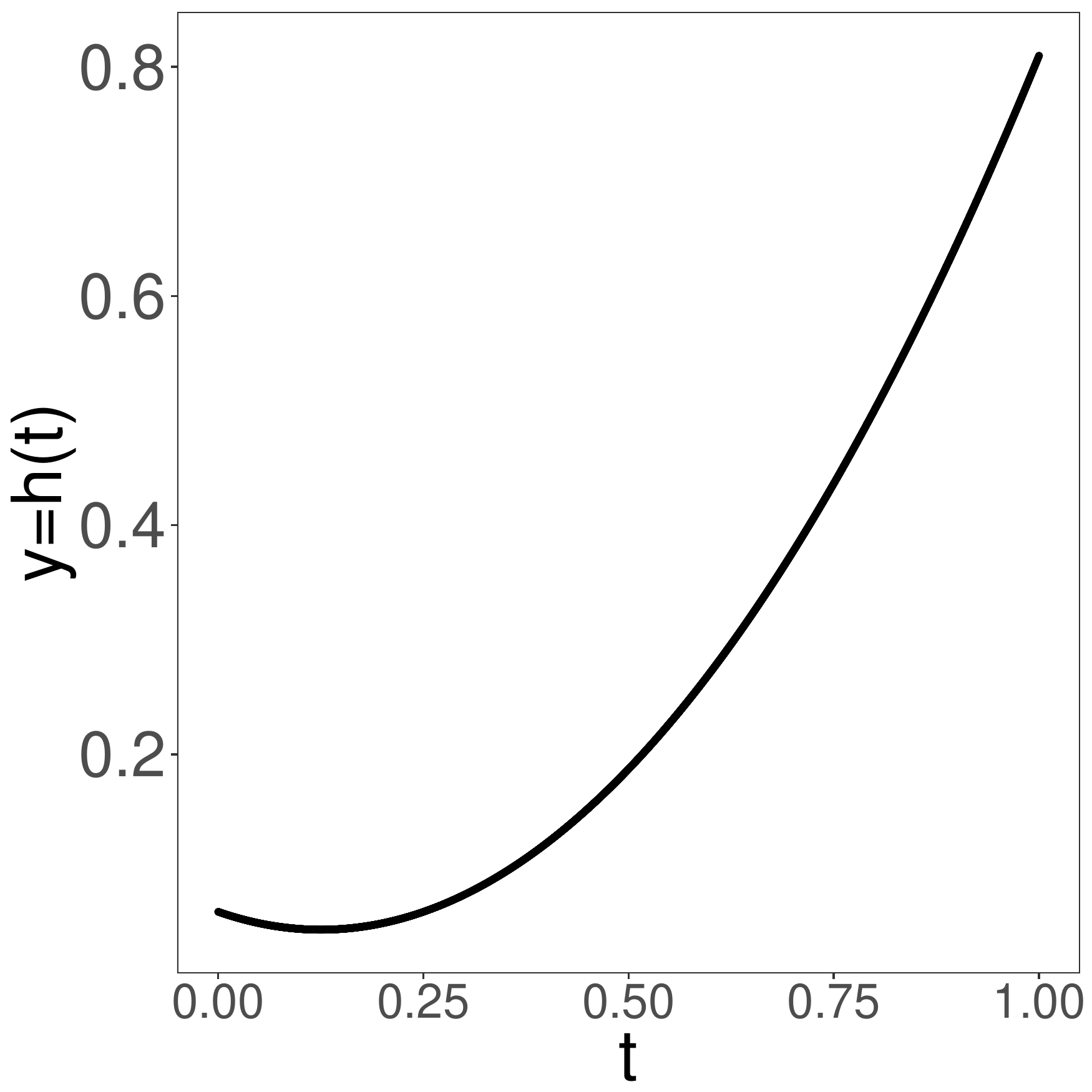}}%
\hspace{0.5in}
 \subfigure[$\mathrm{I}(h_{7})=\mathrm{I}_n(h_{7})=8157$]{%
    \includegraphics[height=0.3\textwidth]{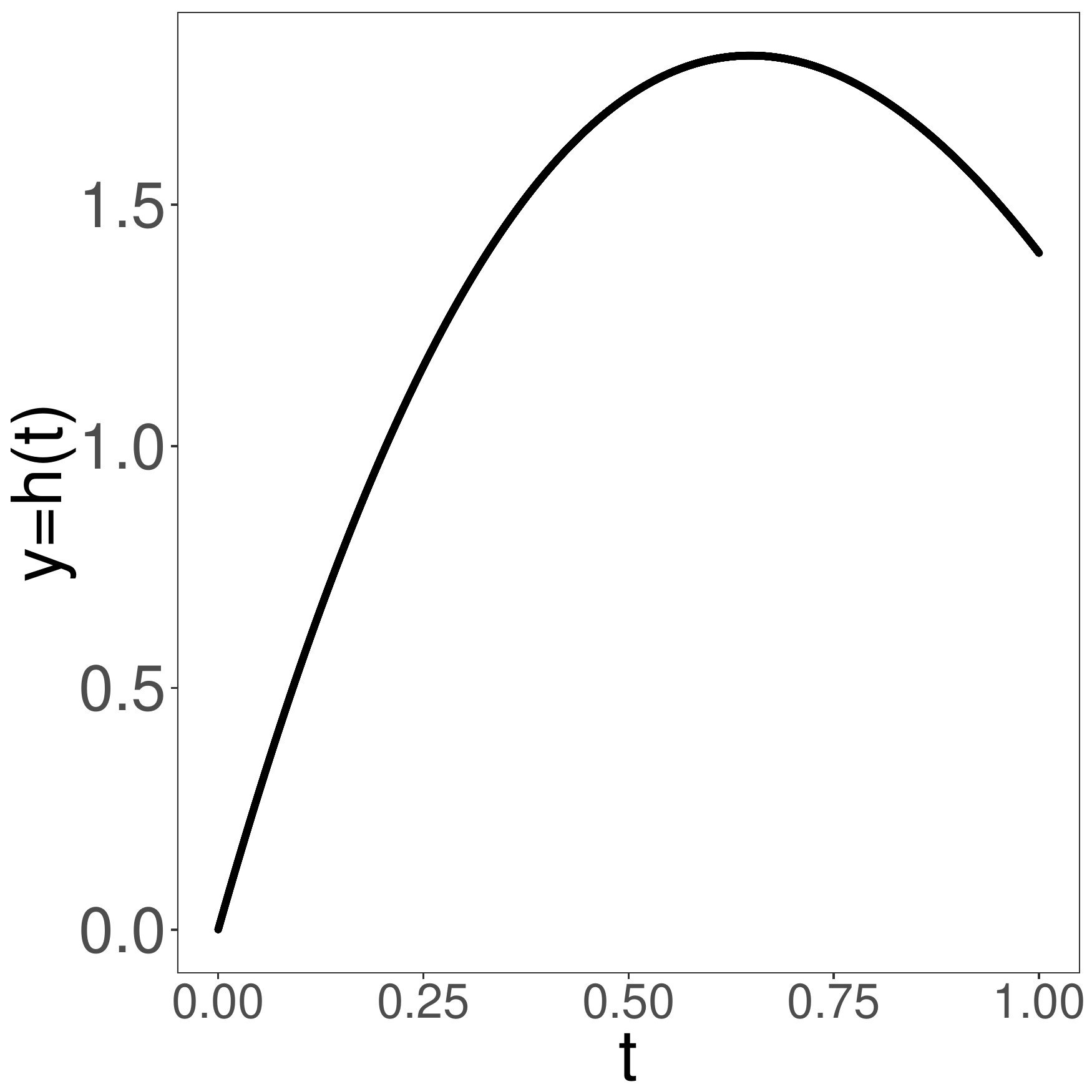}}%
\hspace{0.5in}
\subfigure[$\mathrm{I}(h_{8})=\mathrm{I}_n(h_{8})=0.5000$]{%
    \includegraphics[height=0.3\textwidth]{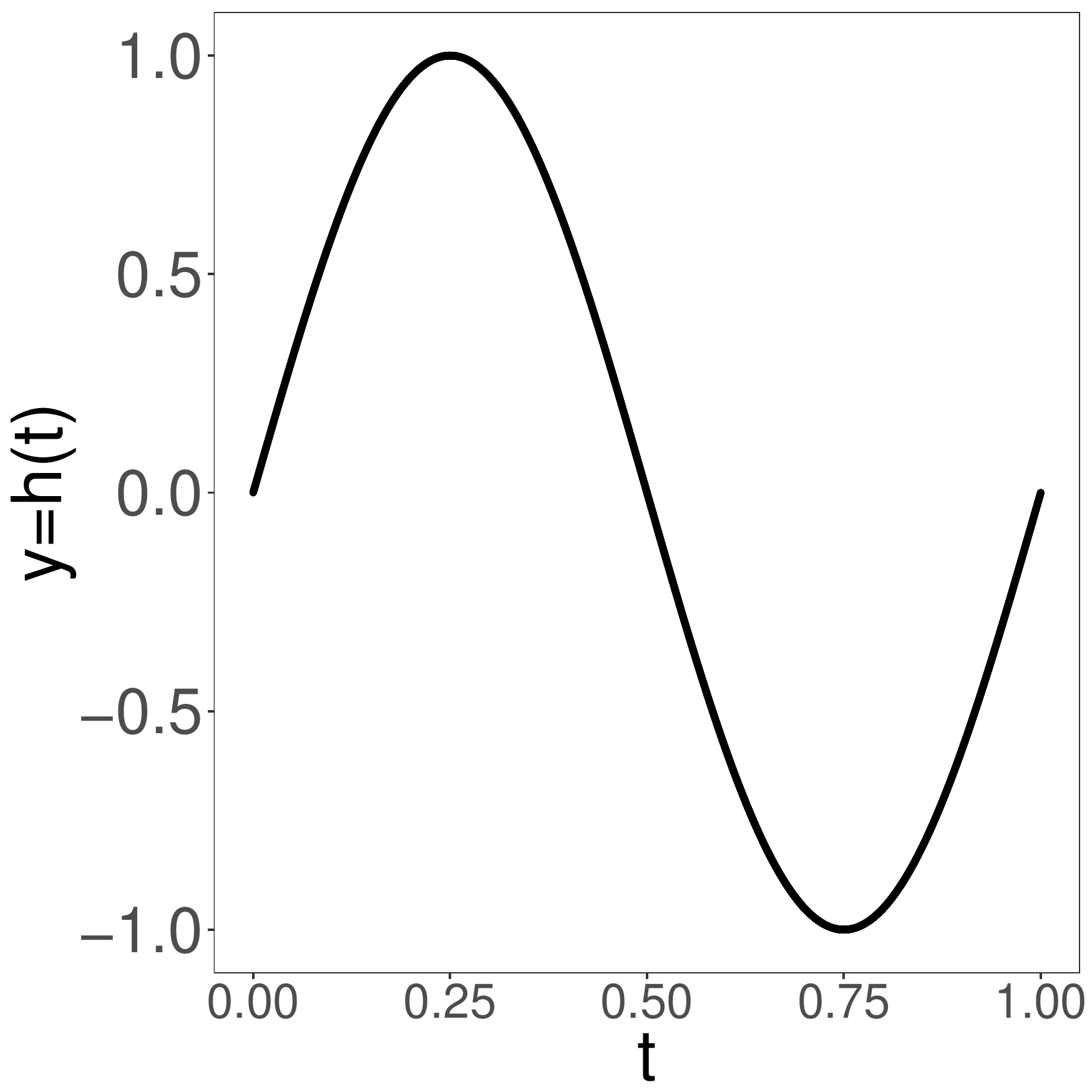}}%
\caption{Quartet (\ref{quartet-2}) functions and their indices of increase}
\label{fig-quartet-2}
\end{figure}
As expected, our preliminary analysis has shown that the un-groped estimators converge to $0.5$ in all the four cases, but the grouped estimator $\widetilde{\mathrm{I}}_{n,\alpha}(h,\varepsilon)$ does converge under appropriate choices of the grouping (or smoothing) parameter $\alpha $ values. Next are summaries of our findings using two approaches: the first one is data exploratory (visual) and the second one is based on cross validation.

\subsection{Data exploratory (visual) choice of $\alpha $}

Based on the crossings of surfaces and hyperplanes depicted in Figure \ref{h5h6h7h8-group},
\begin{figure}[h!]
  \centering
  \subfigure[The hyperplane at the height $\mathrm{I}(h_5)=0.3311$]{%
    \includegraphics[width=0.5\textwidth]{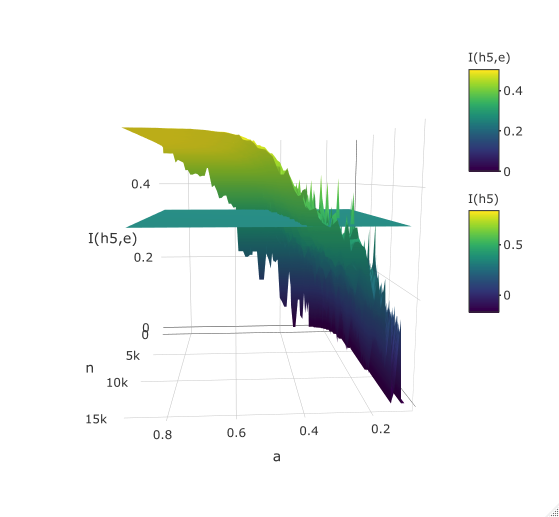}}%
\subfigure[The hyperplane at the height $\mathrm{I}(h_6)=0.9799$]{%
    \includegraphics[width=0.5\textwidth]{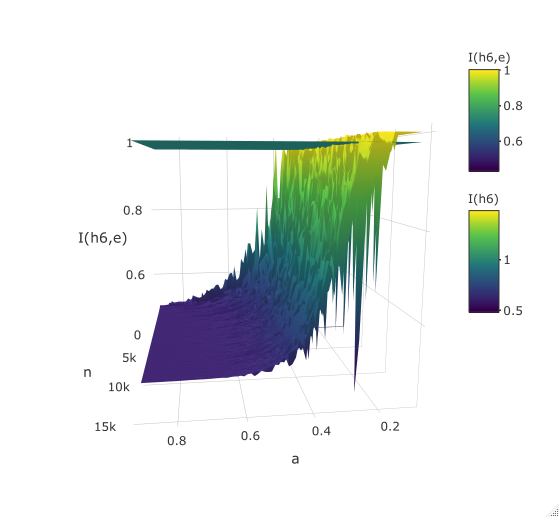}}%
    \\
 \subfigure[The hyperplane at the height $\mathrm{I}(h_7)=0.8157$]{%
    \includegraphics[width=0.5\textwidth]{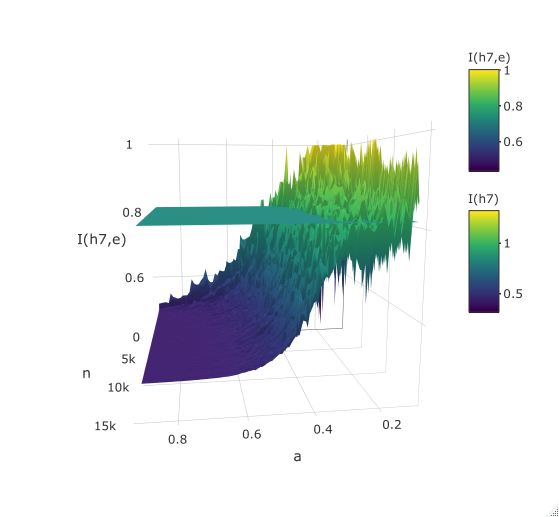}}%
\subfigure[The hyperplane at the height $\mathrm{I}(h_8)=0.5000$]{%
    \includegraphics[width=0.5\textwidth]{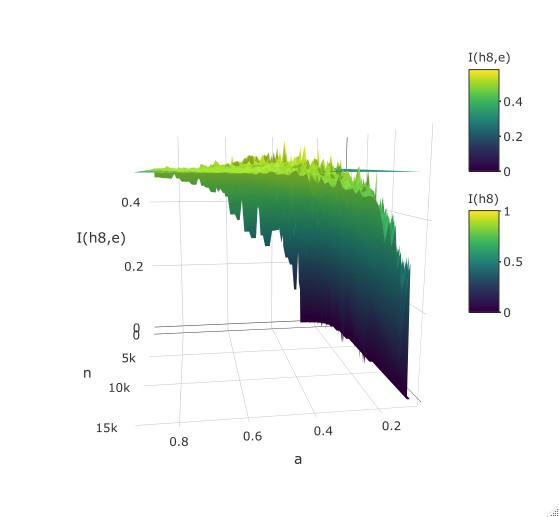}}%
    \caption{Values of $\widetilde{\mathrm{I}}_{n,\alpha}(h,\varepsilon)$ with respect to $n$ and $\alpha$ in the case of quartet (\ref{quartet-2}).}%
\label{h5h6h7h8-group}
\end{figure}	
we choose appropriate $\alpha$ values, denoted by $\alpha_{\text{vi}}$, for the functions of quartet (\ref{quartet-2}). To check the performance of these values, we draw convergence graphs in Figure \ref{h5h6h7h8-group-fixed}.
\begin{figure}[h!]
  \centering
  \subfigure[$\mathrm{I}(h_5)=0.3311$, $\alpha_{\text{vi}}=0.36$]{%
    \includegraphics[height=0.35\textwidth]{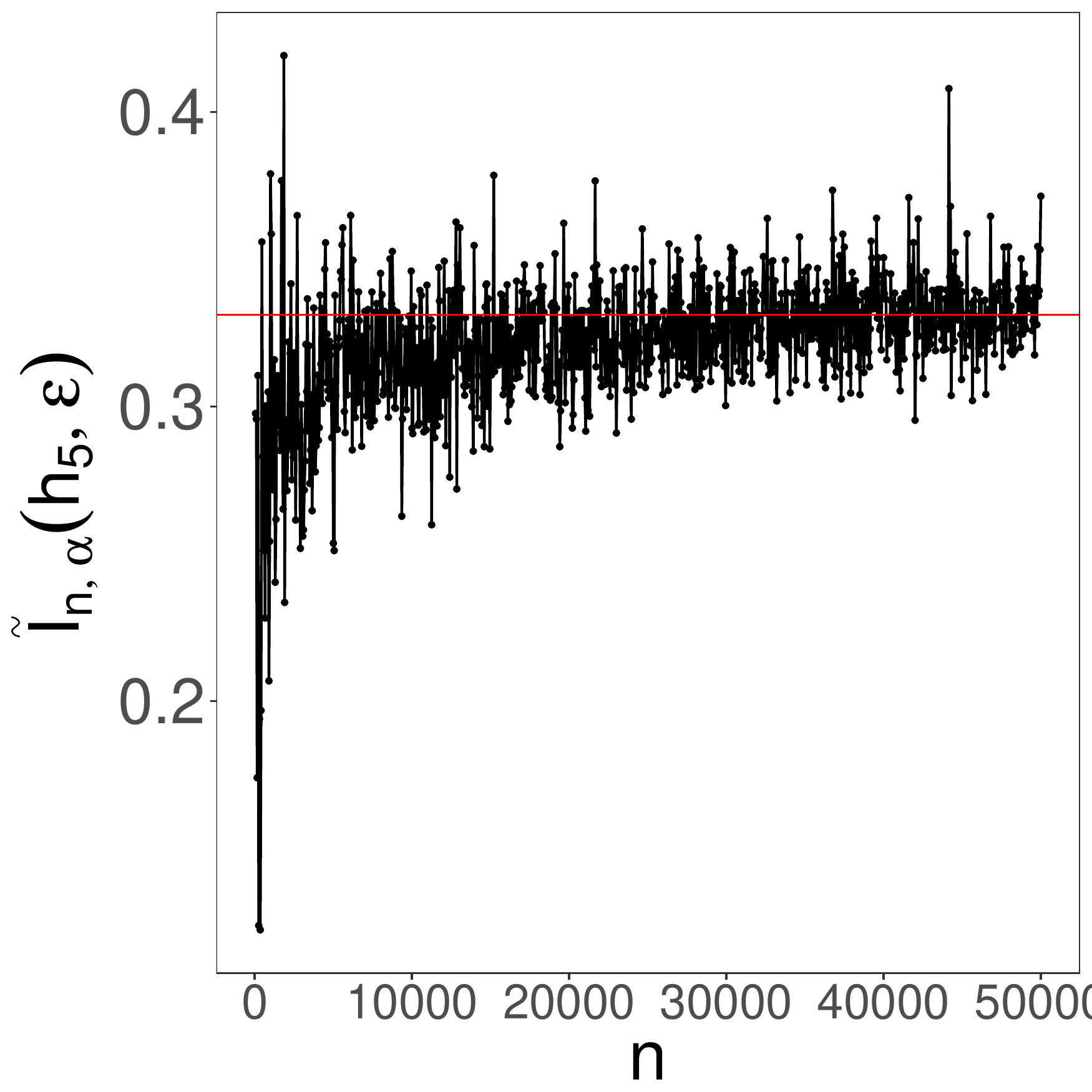}}%
\hspace{0.5in}
\subfigure[$\mathrm{I}(h_6)=0.9799$, $\alpha_{\text{vi}}=0.25$]{%
    \includegraphics[height=0.35\textwidth]{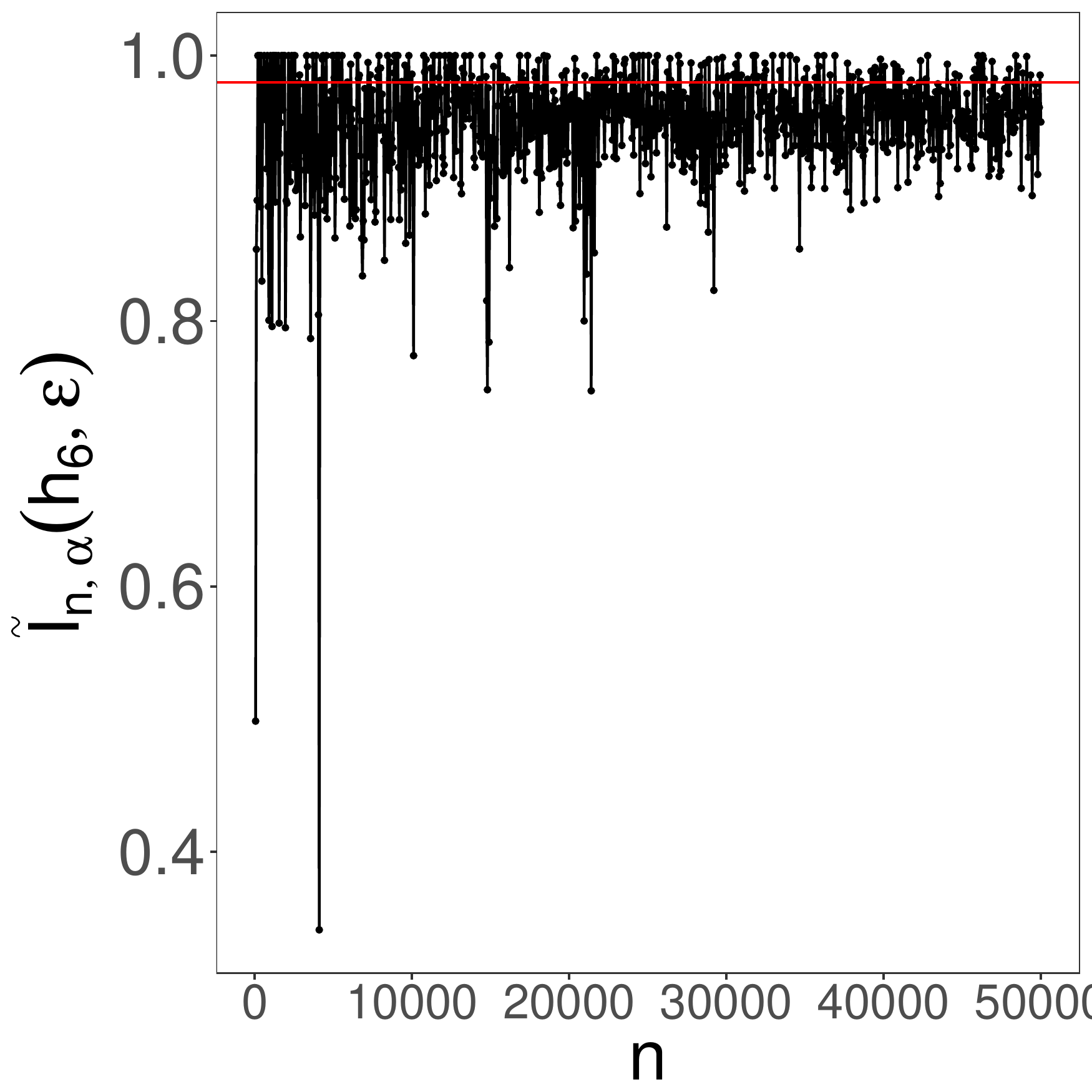}}%
\hspace{0.5in}
 \subfigure[$\mathrm{I}(h_7)=0.8157$, $\alpha_{\text{vi}}=0.28$]{%
    \includegraphics[height=0.35\textwidth]{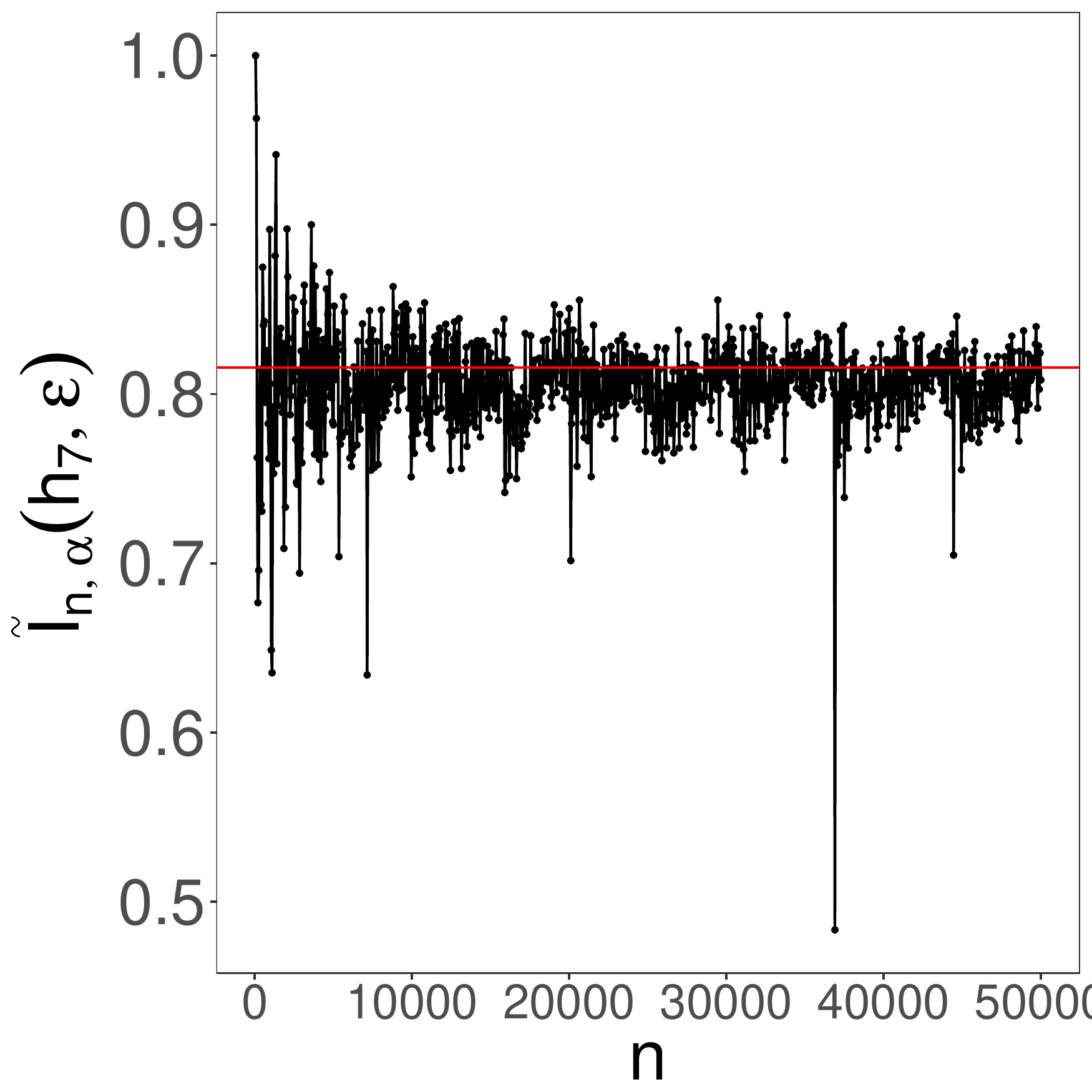}}%
\hspace{0.5in}
\subfigure[$\mathrm{I}(h_8)=0.5000$, $\alpha_{\text{vi}}=0.50$]{%
    \includegraphics[height=0.35\textwidth]{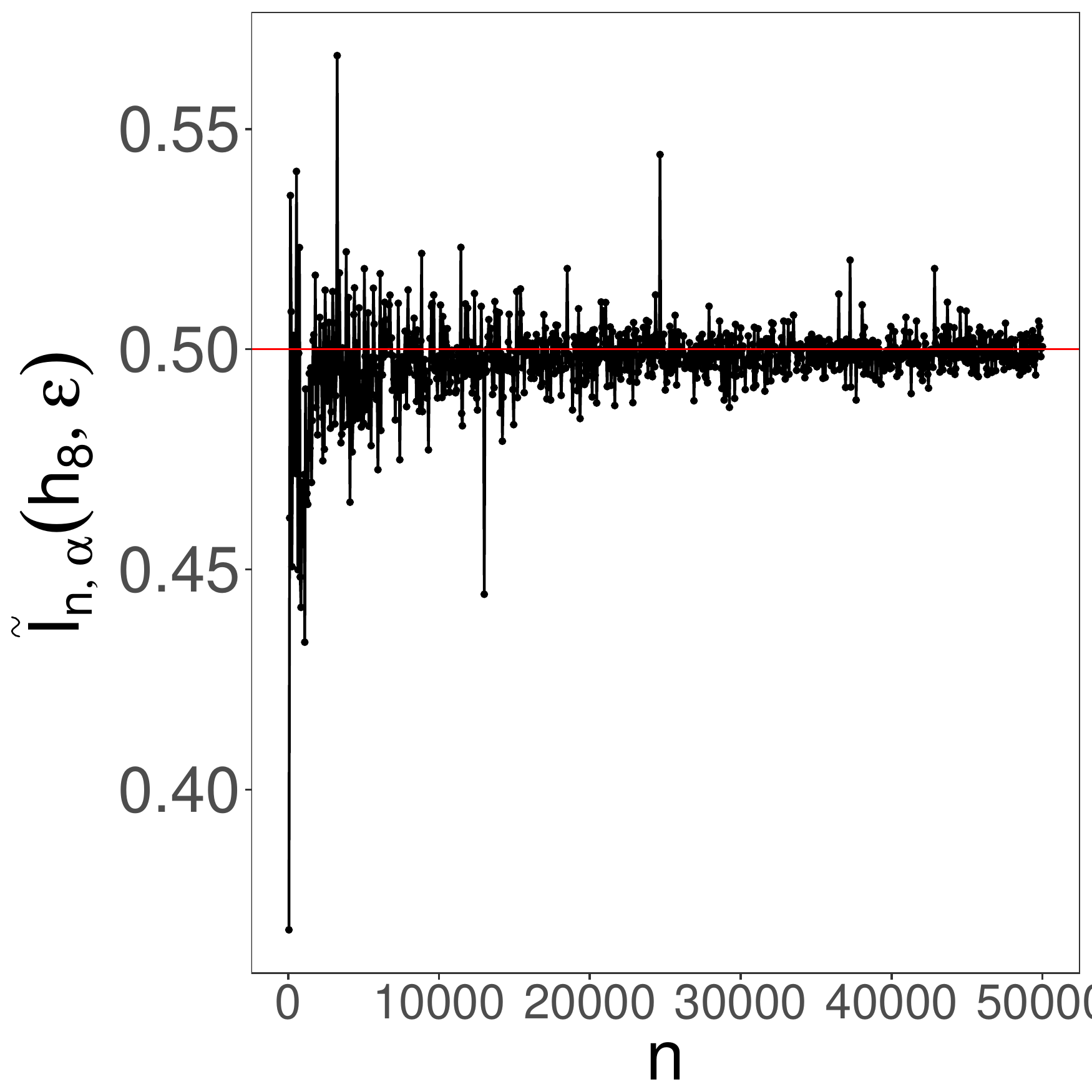}}%
    \caption{The performance of $\widetilde{\mathrm{I}}_{n,\alpha}(h,\varepsilon)$ with respect to $n$ in the case of quartet (\ref{quartet-2}) and based on visually assessed $\alpha$'s.}%
\label{h5h6h7h8-group-fixed}
\end{figure}	
Next, we use formula (\ref{approx-random-grouped}) to calculate point estimates of the actual index of increase for each of the functions in quartet (\ref{quartet-2}), whose values appear in Table \ref{table-1}.
\begin{table}[h!]\small
\centering
\begin{tabular}{l|cccc}
\hline\hline
	   & $h_{5}$ & $h_{6}$ &$h_{7}$ & $h_{8}$\\
\hline
    True values & 0.3311 & 0.9799 &0.8157&0.5000\\
    Point estimates & 0.3274 & 0.9737 &0.8094& 0.5042 \\
    Standard deviations & 0.0745 & 0.1103 & 0.1067 & 0.05237\\
    Confidence intervals &(0.0372, 0.3368)&(0.6527, 1.0000)&(0.6090, 1.0000)&(0.3675, 0.5713)\\
\hline
Estimates $\alpha_{\text{vi}}$& 0.36 & 0.25 & 0.28 & 0.50\\
\hline
\end{tabular}
\caption{Basic statistics and 95\% confidence intervals for quartet (\ref{quartet-2}) based on visually assessed $\alpha$'s.}
\label{table-1}
\end{table}
Finally, we use bootstrap to get standard errors and confidence intervals, all of which are also reported in Table \ref{table-1}.

Reflecting upon the findings in Table \ref{table-1}, we see that the values of $\alpha_{\text{vi}}$ corresponding to the functions $h_5$ and $h_8$ are outside the range $(0,1/3)$ specified by the consistency result of Theorem~\ref{th-1}, but this of course does not invalidate anything -- we are simply working with finite sample sizes $n$. Naturally, we are now eager to compare all the findings reported in Table \ref{table-1} with the corresponding ones obtained by cross validation, which is our next topic.

\subsection{Choosing $\alpha $ based on cross validation}

We now use the cross-validation technique to get estimates $\alpha_{\text{cv}}$ of the grouping parameter $\alpha $ for all the functions of quartet (\ref{quartet-2}). In Figure \ref{h5h6h7h8-cv},
\begin{figure}[h!]
  \centering
  \subfigure[Function $h_5$]{%
    \includegraphics[height=0.35\textwidth]{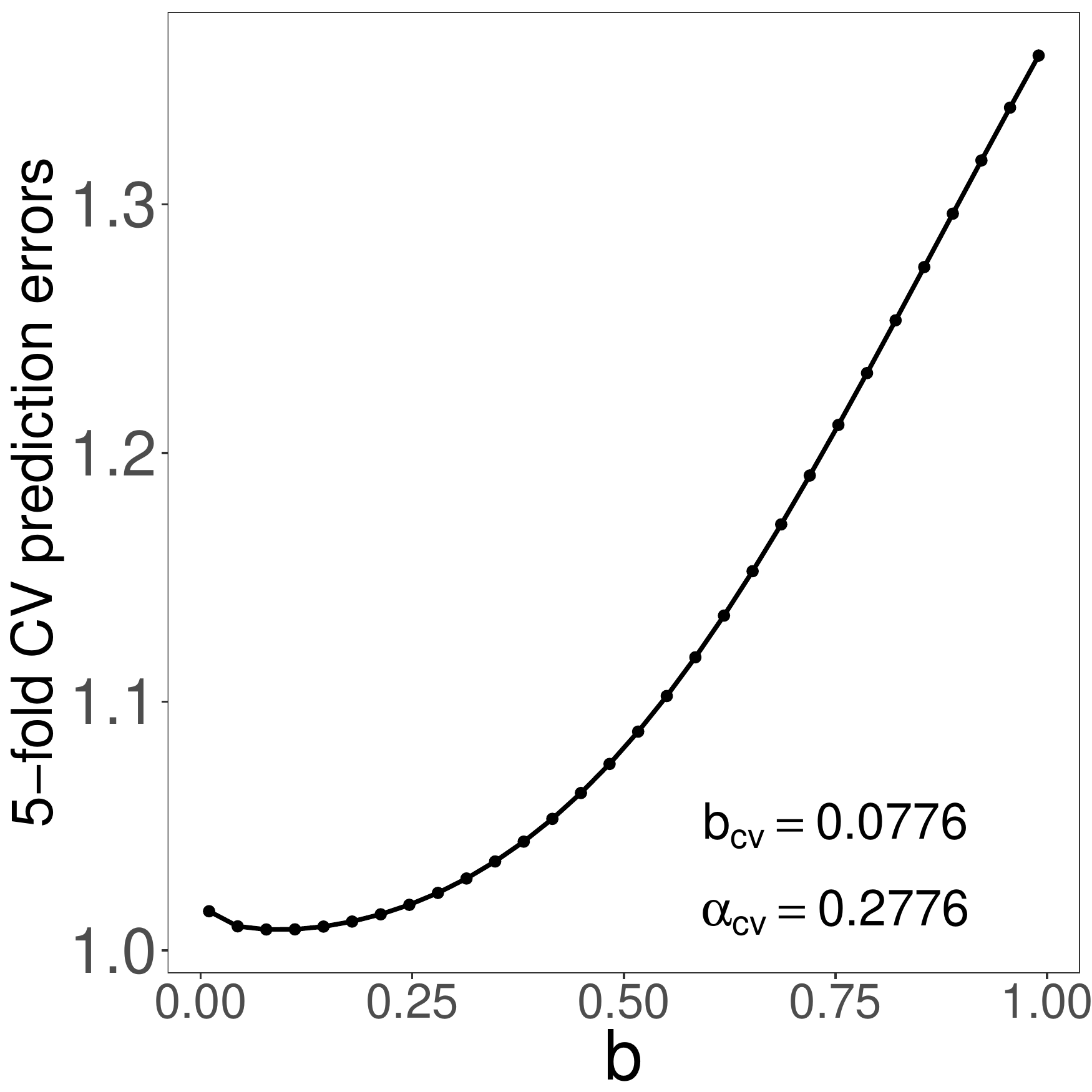}}%
\hspace{0.5in}
\subfigure[Function $h_6$]{%
    \includegraphics[height=0.35\textwidth]{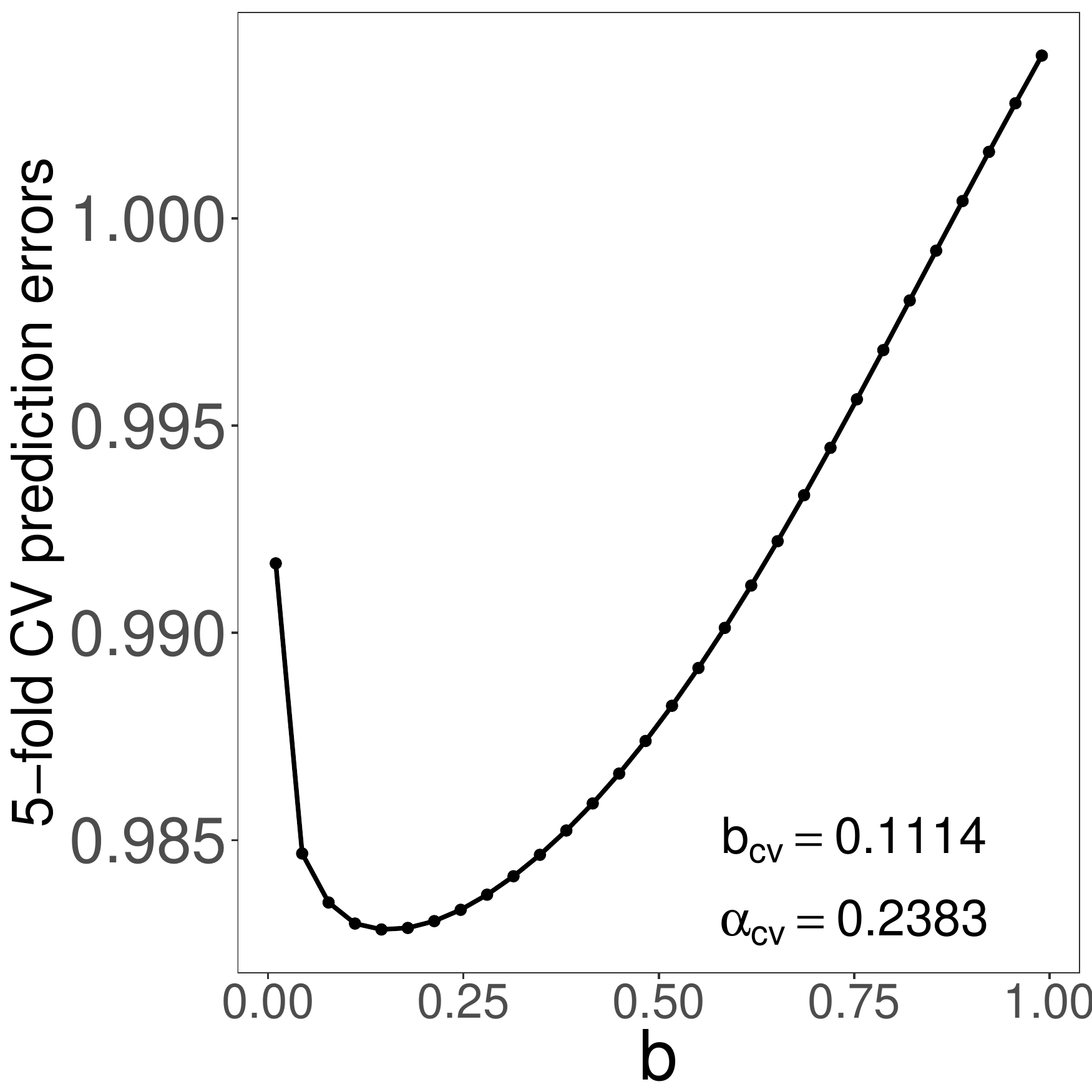}}%
\hspace{0.5in}
 \subfigure[Function $h_7$]{%
    \includegraphics[height=0.35\textwidth]{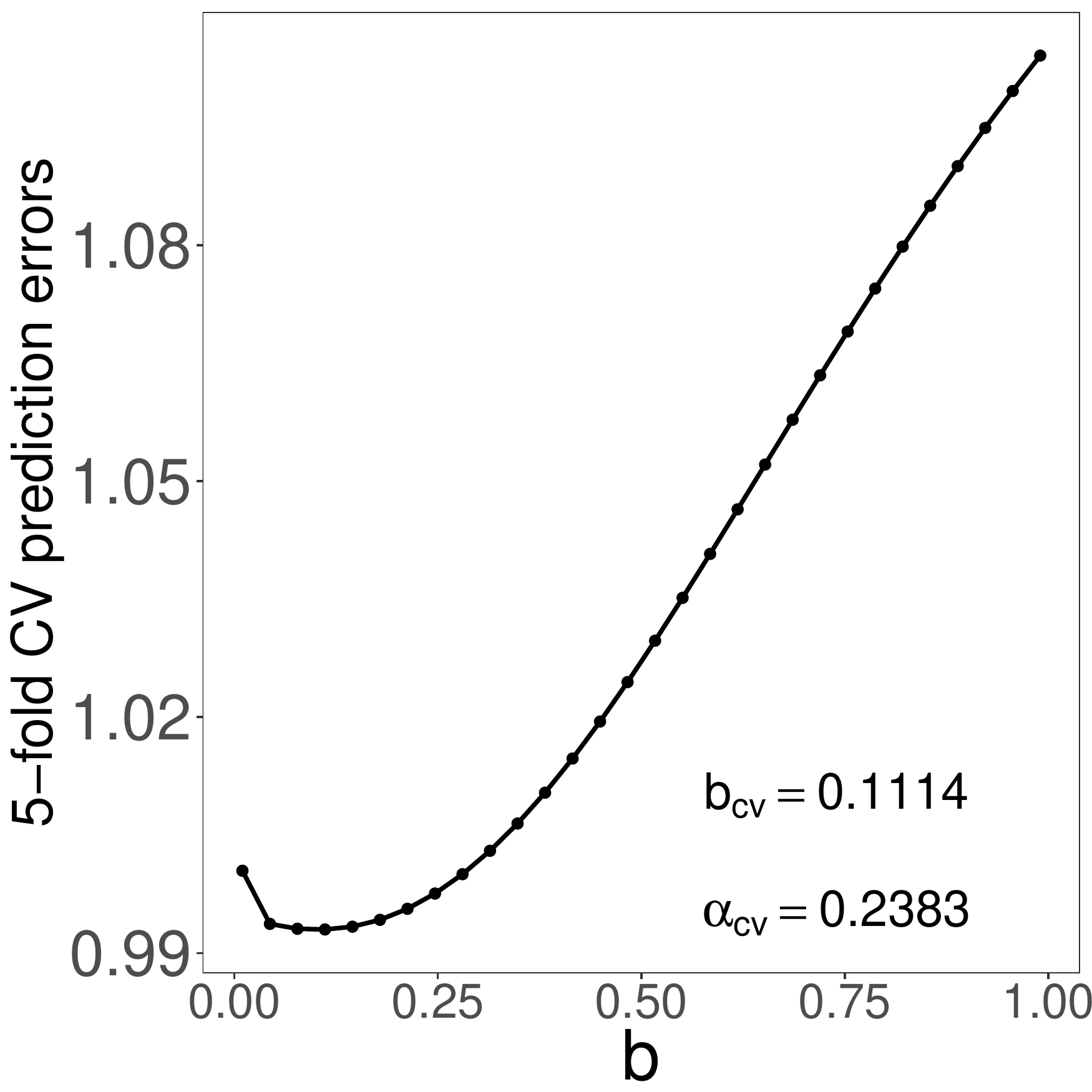}}%
\hspace{0.5in}
\subfigure[Function $h_8$]{%
    \includegraphics[height=0.35\textwidth]{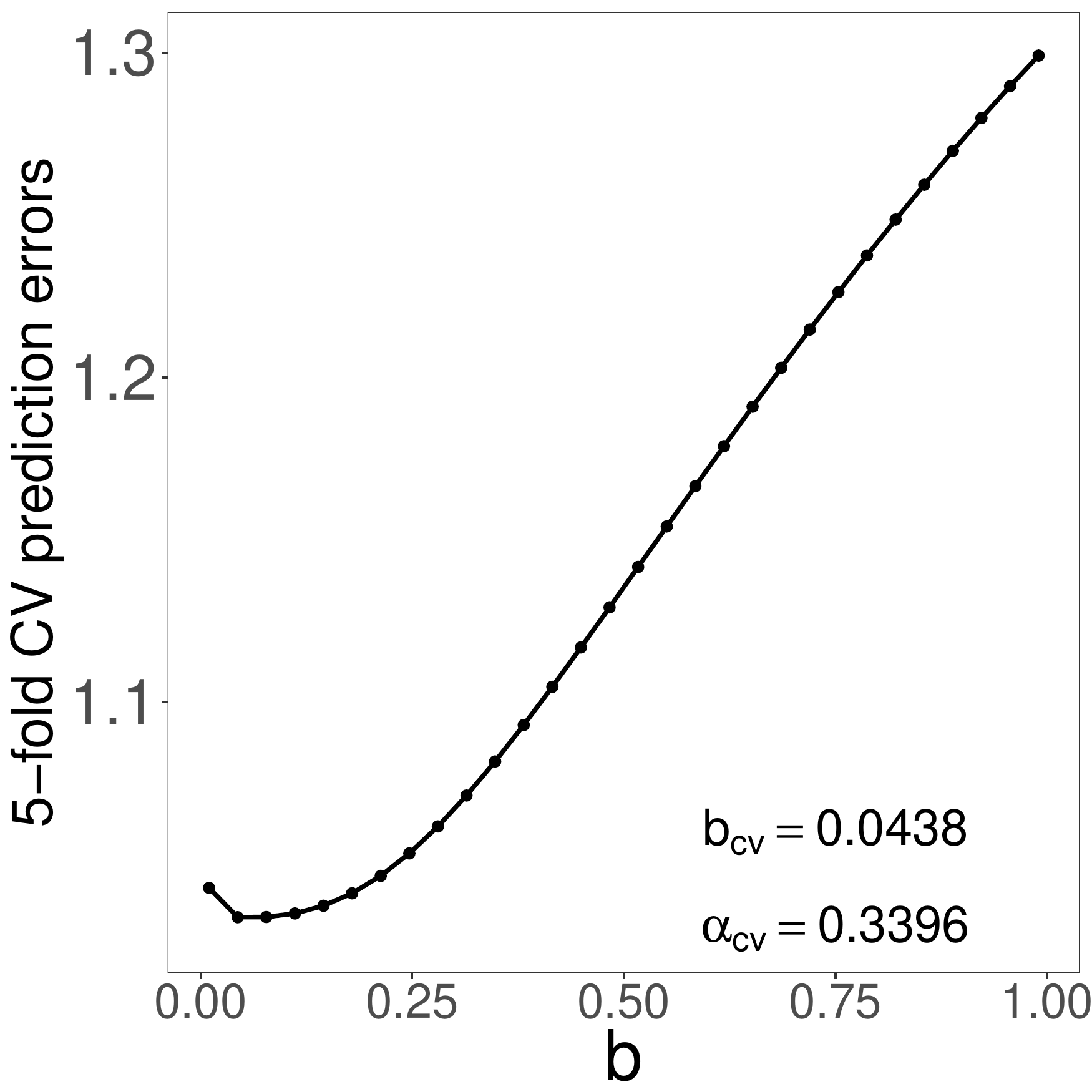}}%
    \caption{Cross validation, minima $b_{\text{cv}}$, and the grouping parameters  $\alpha_{\text{cv}}$ for quartet (\ref{quartet-2}).}%
\label{h5h6h7h8-cv}
\end{figure}	
we visualize the cross-validation scores, specify their minima $b_{\text{cv}}$, and also report the grouping parameters $\alpha_{\text{cv}}$ derived via the equation $ \alpha_{\text{cv}}=\log(1/b_{\text{cv}})/\log(n)$. Based on these $\alpha_{\text{cv}}$ values, we explore the performance of $\widetilde{\mathrm{I}}_{n,\alpha}(h,\varepsilon)$ using the convergence graphs depicted in Figure \ref{h5h6h7h8-group-fixed-cv}.
\begin{figure}[h!]
  \centering
  \subfigure[$\mathrm{I}(h_5)=0.3311$, $\alpha_{\text{cv}}=0.28$]{%
    \includegraphics[height=0.35\textwidth]{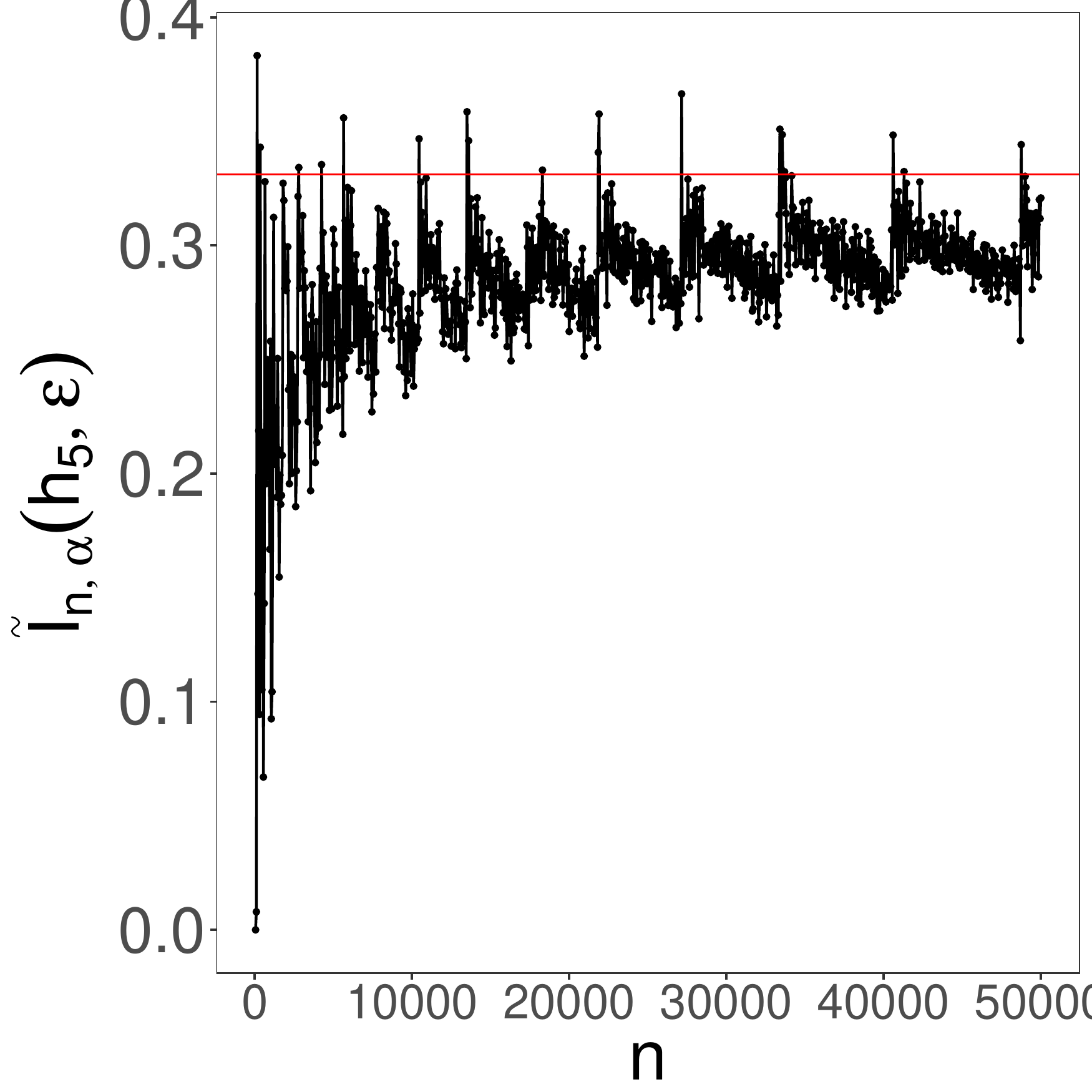}}%
\hspace{0.5in}
\subfigure[$\mathrm{I}(h_6)=0.9799$, $\alpha_{\text{cv}}=0.24$]{%
    \includegraphics[height=0.35\textwidth]{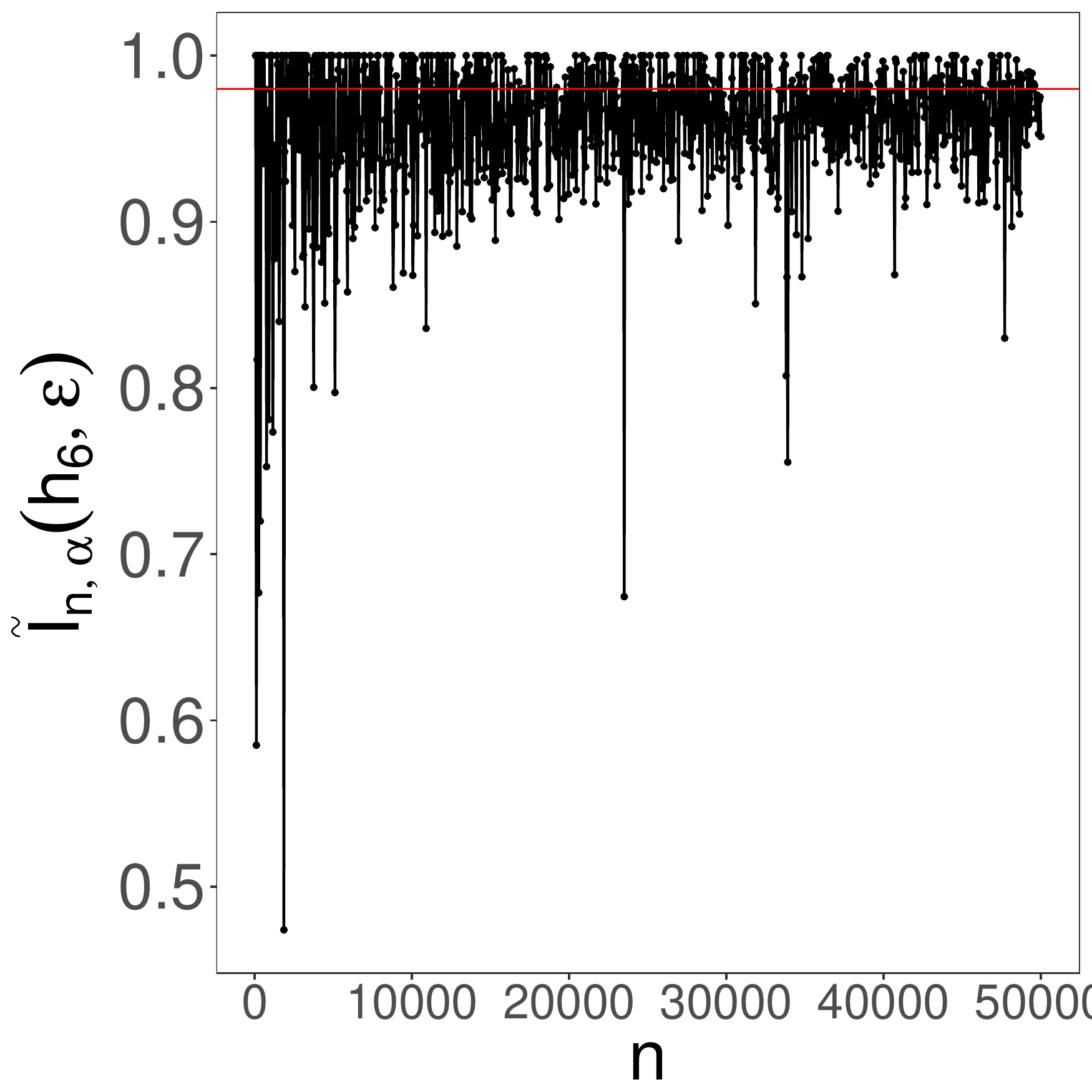}}%
\hspace{0.5in}
 \subfigure[$\mathrm{I}(h_7)=0.8157$, $\alpha_{\text{cv}}=0.24$]{%
    \includegraphics[height=0.35\textwidth]{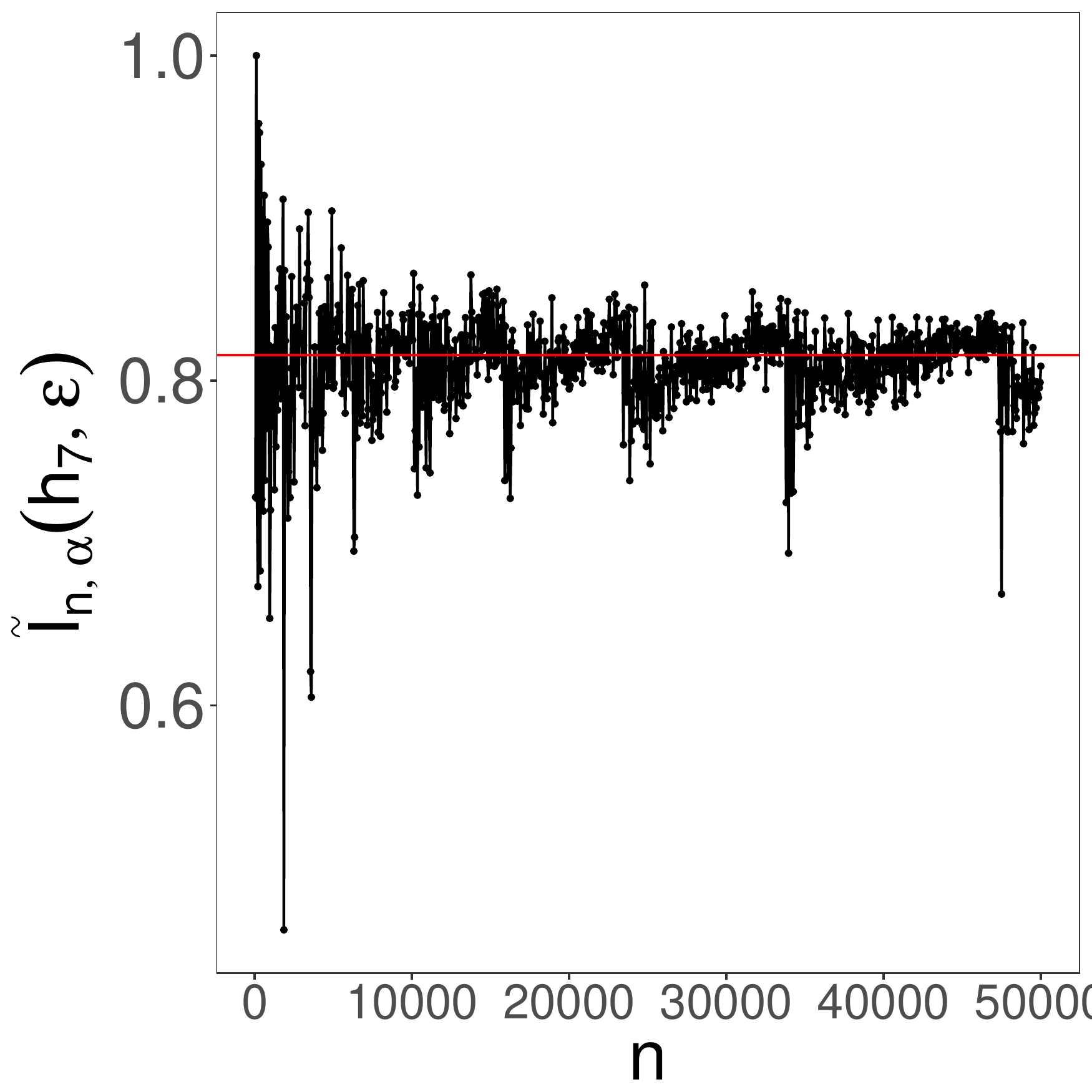}}%
\hspace{0.5in}
\subfigure[$\mathrm{I}(h_8)=0.5000$, $\alpha_{\text{cv}}=0.34$]{%
    \includegraphics[height=0.35\textwidth]{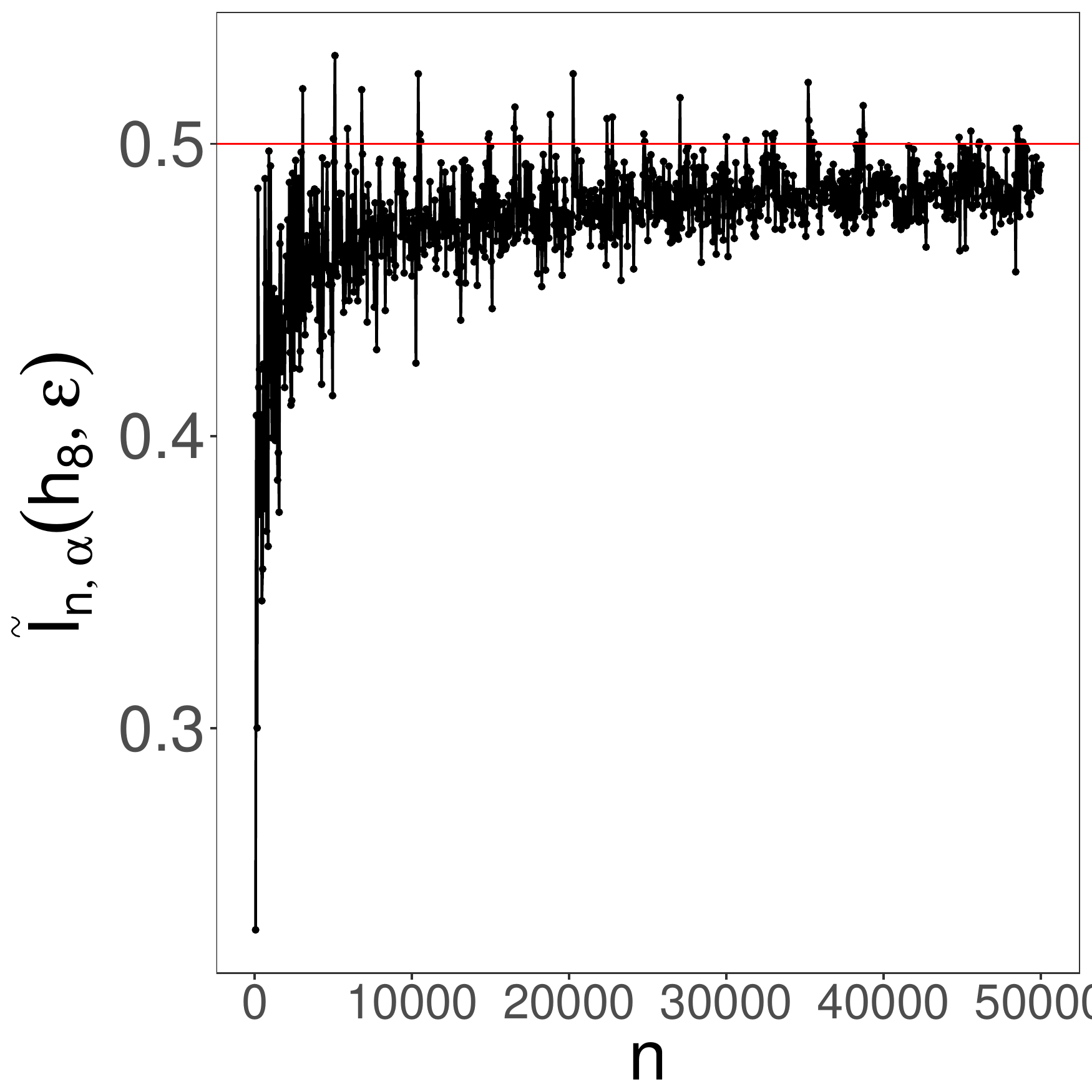}}%
    \caption{The performance of $\widetilde{\mathrm{I}}_{n,\alpha}(h,\varepsilon)$ with respect to $n$ in the case of quartet (\ref{quartet-2}) and cross validated $\alpha$'s.}%
\label{h5h6h7h8-group-fixed-cv}
\end{figure}
The values of point estimates, standard errors, and confidence intervals are reported in Table \ref{table-cv}.
\begin{table}[h!]\small
\centering
\begin{tabular}{l|cccc}
\hline\hline
	   & $h_{5}$ & $h_{6}$ &$h_{7}$ & $h_{8}$\\
\hline
    True values & 0.3311 & 0.9799 &0.8157&0.5000\\
    Point estimates & 0.2771 & 0.9894 &0.8378& 0.4703 \\
    Standard deviations & 0.0797 & 0.1201 & 0.1084 & 0.1185\\
    Confidence intervals &(0.0000, 0.2813)&(0.5987, 1.0000)&(0.6256, 1.0000)&(0.1803, 0.6193)\\
\hline
Estimates $\alpha_{\text{cv}}$& 0.28 & 0.24 & 0.24 & 0.34 \\
\hline
\end{tabular}
\caption{Basic statistics and 95\% confidence intervals for quartet (\ref{quartet-2}) based on cross validation.}
\label{table-cv}
\end{table}

Note that the first three values of $\alpha_{\text{cv}}$ reported in Table~\ref{table-cv} are inside the range $(0,1/3)$ specified by the consistency result of Theorem~\ref{th-1}, whereas $\alpha_{\text{cv}}=0.3396$ corresponding to $h_8$ is just slightly outside the range. Note also that the values of $\alpha_{\text{cv}}$ corresponding to the functions $h_5$ and $h_8$ are considerably smaller than the corresponding $\alpha_{\text{vi}}$'s reported in Table~\ref{table-1}.

The confidence intervals reported in Tables \ref{table-1} and  \ref{table-cv} comfortably cover the actual values of $\mathrm{I}(h)$, and the widths of these confidence intervals, denoted by $\textrm{width}_{\text{vi}}$ and $\textrm{width}_{\text{cv}}$ respectively, are comparable for the functions $h_5$, $h_6$ and $h_7$. The $\textrm{width}_{\text{cv}}$ of the cv-based confidence interval for the function $h_8$ is, however, considerably wider than the corresponding $\textrm{width}_{\text{vi}}$ reported in Table \ref{table-1}. In summary, the relative differences $\textrm{width}_{\text{cv}}/\textrm{width}_{\text{vi}}-1$ for the functions $h_5$, $h_6$, $h_7$ and $h_8$ are $-0.0611$, $0.1555$, $-0.0425$ and $1.1541$, respectively. We finish the discussion by recalling Wasserman's (2005) advice: ``Do not assume that, if the estimator [...] is wiggly, then cross-validation has let you down. The eye is not a good judge of risk'' (Remark~20.18, page~317).

\section{Summary and concluding notes}
\label{conclude}

Davydov and Zitikis (2017) introduced an index of increase when populations are modelled with continuous functions. Chen and Zitikis (2017) explored a modification of the index when populations are discrete and presented in the form of scatterplots, and they also explored the situation when scatterplots are viewed as data sets, in which case they fitted (non-monotonic) regression functions and subsequently applied the technique by Davydov and Zitikis (2017) to assess monotonicity of the fitted functions.

In the present paper we have extended the aforementioned technique to the case when it is not desirable, or appropriate, to view scatterplots as populations, or to use regression methods to fit curves to scatterplots. The herein proposed technique is based on grouping and averaging data, and then calculating the index of increase. Since the grouping parameter depends on both deterministic and random features of the underlying problem, we have suggested a way for grouping data so that the resulting estimator of the index of increase would be consistent. Based on this estimator, we have then suggested a construction of bootstrap-based confidence intervals for the index of increase.

The derived theoretical results have been made accessible to practitioners by detailed descriptions and analyses of various computational aspects inherent in our proposed solution of the problem.

\section*{Acknowledgements}

We are indebted to the anonymous reviewers for suggestions, insightful comments, and constructive criticism that guided our work on the revision. The research has been supported by the Natural Sciences and Engineering Research Council (NSERC) of Canada.

\end{document}